\documentclass[11pt]{article}%
\usepackage[font=small,format=plain,labelfont=bf,up]{caption}
\usepackage[a4paper,nohead,ignorefoot,%
twoside,scale=0.85,bindingoffset=1.25cm]{geometry}
\usepackage{amssymb,amsfonts,amsmath,amsthm}
\usepackage[]{graphicx}
\usepackage{color}
\graphicspath{{figures/}}
\newcommand{\url}[1]{#1} 
\usepackage{hyperref}%
\definecolor{gray}{rgb}{0.2,0.2,.2}
\hypersetup{%
  unicode=true,          
  colorlinks=true,       
  linkcolor=gray,          
  citecolor=gray      
}
\definecolor{colorBlue}{rgb}{0.,.0,0.825}
\definecolor{colorGreen}{rgb}{0.,0.75,.0}
\definecolor{colorRed}{rgb}{0.99,0.,.0}

\newcommand{\bpm}{\begin{pmatrix}}
\newcommand{\epm}{\end{pmatrix}}

\DeclareMathOperator{\Id}{Id}

\newcommand{\sgn}{\mathrm{sgn}}

\newcommand{\fspace}[1]{{\mathsf{#1}}}
\newcommand{\fspaceL}{\fspace{L}}

\newcommand{\fspaceC}{\fspace{C}}
\newcommand{\fspaceW}{\fspace{W}}

\newcommand{\ol}[1]{{\overline{#1}}}

\newcommand{\Rset}{{\mathbb{R}}}

\newcommand{\Nset}{{\mathbb{N}}}

\newcommand{\cointerval}[2]{[#1,\,#2)}%
\newcommand{\oointerval}[2]{(#1,\,#2)}%
\newcommand{\ccinterval}[2]{[#1,\,#2]}%

\newcommand{\loc}{{\rm loc}}

\newcommand{\ini}{{\rm ini}}

\newcommand{\ess}{{\rm ess}}
\newcommand{\negl}{{\rm neg}}

\newcommand{\pair}[2]{{\left({#1},\,{#2}\right)}}

\newcommand{\bpair}[2]{{\big({#1},\,{#2}\big)}}
\newcommand{\Bpair}[2]{{\Big({#1},\,{#2}\Big)}}

\newcommand{\at}[1]{{\left({#1}\right)}}
\newcommand{\nat}[1]{(#1)}
\newcommand{\bat}[1]{{\big(#1\big)}}
\newcommand{\Bat}[1]{{\Big(#1\Big)}}

\newcommand{\triple}[3]{{\left({#1},\,{#2},\,{#3}\right)}}

\newcommand{\D}{\displaystyle}

\newcommand{\wh}[1]{{\widehat{#1}}}

\newcommand{\bjump}[1]{{\big|\!\big[#1\big]\!\big|}}

\newcommand{\tnorm}[1]{{\left\vert\kern-0.25ex\left\vert\kern-0.25ex\left\vert #1 
    \right\vert\kern-0.25ex\right\vert\kern-0.25ex\right\vert}}
\newcommand{\norm}[1]{\left\|{#1}\right\|}

\newcommand{\bnorm}[1]{\big\|{#1}\big\|}
\newcommand{\Bnorm}[1]{\Big\|{#1}\Big\|}
\newcommand{\abs}[1]{\left|{#1}\right|}

\newcommand{\babs}[1]{\big|{#1}\big|}

\newcommand{\dint}[1]{\,\mathrm{d}#1}

\newcommand{\al}{{\alpha}}

\newcommand{\ga}{{\gamma}}

\newcommand{\eps}{{\varepsilon}}

\theoremstyle{plain}
\newtheorem{theorem}             {Theorem}[]
\newtheorem{corollary}  [theorem]{Corollary}

\newtheorem{lemma}      [theorem]{Lemma}

\newtheorem{scheme}     [theorem]{Scheme}
\newtheorem*{result*}{Main result}
\theoremstyle{definition}
\newtheorem{definition} [theorem]{Definition}
\newtheorem{assumption} [theorem]{Assumption}
\newtheorem*{remarks*}{Remarks}
\newtheorem*{remark*}{Remark}
\usepackage{float}
\begin{document}
%
%
%
%
\title{Hysteretic dynamics of phase interfaces \\in bilinear forward-backward diffusion equations }
\date{\today}
\author{%
Michael Herrmann\footnote{Technische Universit\"at Braunschweig, Germany, {\tt michael.herrmann@tu-braunschweig.de}}\and
Dirk Jan{\ss}en\footnote{Technische Universit\"at Braunschweig, Germany, {\tt d.janssen@tu-braunschweig.de}}
 }
\maketitle
%
%
%
\begin{abstract}
We study single-interface solutions to a free boundary problem that couples bilinear bulk diffusion to the Stefan condition and a hysteretic flow rule for phase boundaries. We introduce a time-discrete approximation scheme  and establish its convergence in the limit of vanishing step size. The main difficulty in our proof are strong microscopic oscillations which are produced by a propagating phase interface and need to be controlled on the macroscopic scale. We also present numerical simulations and discuss the link to the viscous regularization of ill-posed diffusion equations.
\end{abstract}
%
%
%
\quad\newline\noindent%
\begin{minipage}[t]{0.15\textwidth}%
Keywords:
\end{minipage}%
\begin{minipage}[t]{0.8\textwidth}%
\emph{hysteretic phase interfaces},  
\emph{free boundary problems},
\\%
\emph{ill-posed diffusion equations, bilinear constitutive laws}
\end{minipage}%
\medskip
\newline\noindent
\begin{minipage}[t]{0.15\textwidth}%
MSC (2010): %
\end{minipage}%
\begin{minipage}[t]{0.8\textwidth}%
35R25,  
35R35,  
74N30 
\end{minipage}%
%
%
%
\setcounter{tocdepth}{3}
\setcounter{secnumdepth}{3}{\scriptsize{\tableofcontents}}
%
%
%
%
%
\section{Introduction}
%
%
Ill-posed diffusion equations have many applications in science and technology and the prototypical equation in a one-dimensional setting can be written as
\begin{align}
\label{eq: PDE}
\partial_t u\pair{t}{x}&=\partial_x^2\Phi^\prime\bat{u\pair{t}{x}}\,.
\end{align}
Here, $u$ is a scalar quantity that depends on time $t\geq0$ and the space variable $x\in\Rset$. In this paper we are interested in the bistable case in which $\Phi^\prime$ is the derivative of a double-well-potential $\Phi$ as illustrated in the left panel of Figure \ref{fig: abl_double_well}. In this setting, equation \eqref{eq: PDE} or its equivalent reformulation
\begin{align*}
\partial_t w\pair{t}{x} = \partial_x\Phi^\prime\at{\partial_x w\pair{t}{x}}
\end{align*}
can be used to model the dynamics of phase transitions in fluids and solids or the flow through porous media as described in \cite{Ell85,GM17}. The two convexity intervals of $\Phi$  represent the phases and correspond to the increasing (or stable) branches of $\Phi^\prime$ which are naturally related to the notion of local forward diffusion. The decreasing (or unstable) branch of $\Phi^\prime$, however, reflects the backward diffusion inside the spinodal region. The monostable case is also important but exhibits a different dynamical behavior, namely the formation and subsequent coarsening of localized patterns, see for instance \cite{Pad04,PM90,HPO04,ES08,EG09}.
\par
The Cauchy problem for \eqref{eq: PDE} is ill-posed due to the non-monotonicity of $\Phi^\prime$. To obtain a well-posed equation, several regularizations have been proposed in the literature and the most prominent example is the one-dimensional Cahn-Hilliard equation
\begin{align}
\label{eq: CH}
\partial_t u\pair{t}{x}=\partial_x^2 \Phi^\prime\at{u\pair{t}{x}}-\nu^2 \partial_x^4 u\pair{t}{x}\,.
\end{align}
More relevant for our work are the viscous approximation
\begin{align}
\label{eq: visk Reg}
\partial_t u\pair{t}{x}-\nu^2\partial_x^2
\partial_t u\pair{t}{x}
=\partial_x^2 \Phi^\prime\bat{u\pair{t}{x}}
\end{align}
as well as the lattice approximation 
\begin{align}
\label{eq: lattice}
\partial_t u\pair{t}{x}=\nu^{-2}\Bat{\Phi^\prime\bat{u\pair{t}{x-\nu}}-2\, \Phi^\prime\bat{u\pair{t}{x}}+ \Phi^\prime\bat{u\pair{t}{x+\nu}}}
\end{align}
since both can be used to justify the hysteric flow rule for phase interfaces that is studied below.
\par
Another classical approach is to study free boundary problems. In this setting, one seeks weak solutions $u$ to \eqref{eq: PDE} that are confined to either one of the two phase everywhere with the exception of a finite number of curves $x=\xi_i\at{t}$. These \emph{phase interfaces} are not given a priori but satisfy certain dynamical conditions which reflect  microscopic details of the phase transitions and complement the PDE. Both approaches or closely related and the complete set of interface conditions is often derived from a regularized model in the limit $\nu\to0$. Of course, one has to prove that sharp sharp interfaces arise for $\nu=0$ and must identify classes of admissible initial data for $\nu>0$.
\par
In this paper we propose a time-discrete scheme with step size $\Delta t>0$ for the computation of single-interface solutions to a free boundary problem with a hysteric interface conditions that is naturally related to the sharp interface limit $\nu\to0$ of both the viscous model \eqref{eq: visk Reg} and the lattice \eqref{eq: lattice}. Moreover, we prove the convergence of the scheme in the limit $\Delta t\to0$ for a bilinear function $\Phi^\prime$ and a certain class of initial data. The latter gives rise to a single phase interface which can alter its propagation mode due to pinning or depinning events.
%
%
\begin{figure}[ht!]
\centering{
\includegraphics[width=0.4\textwidth]{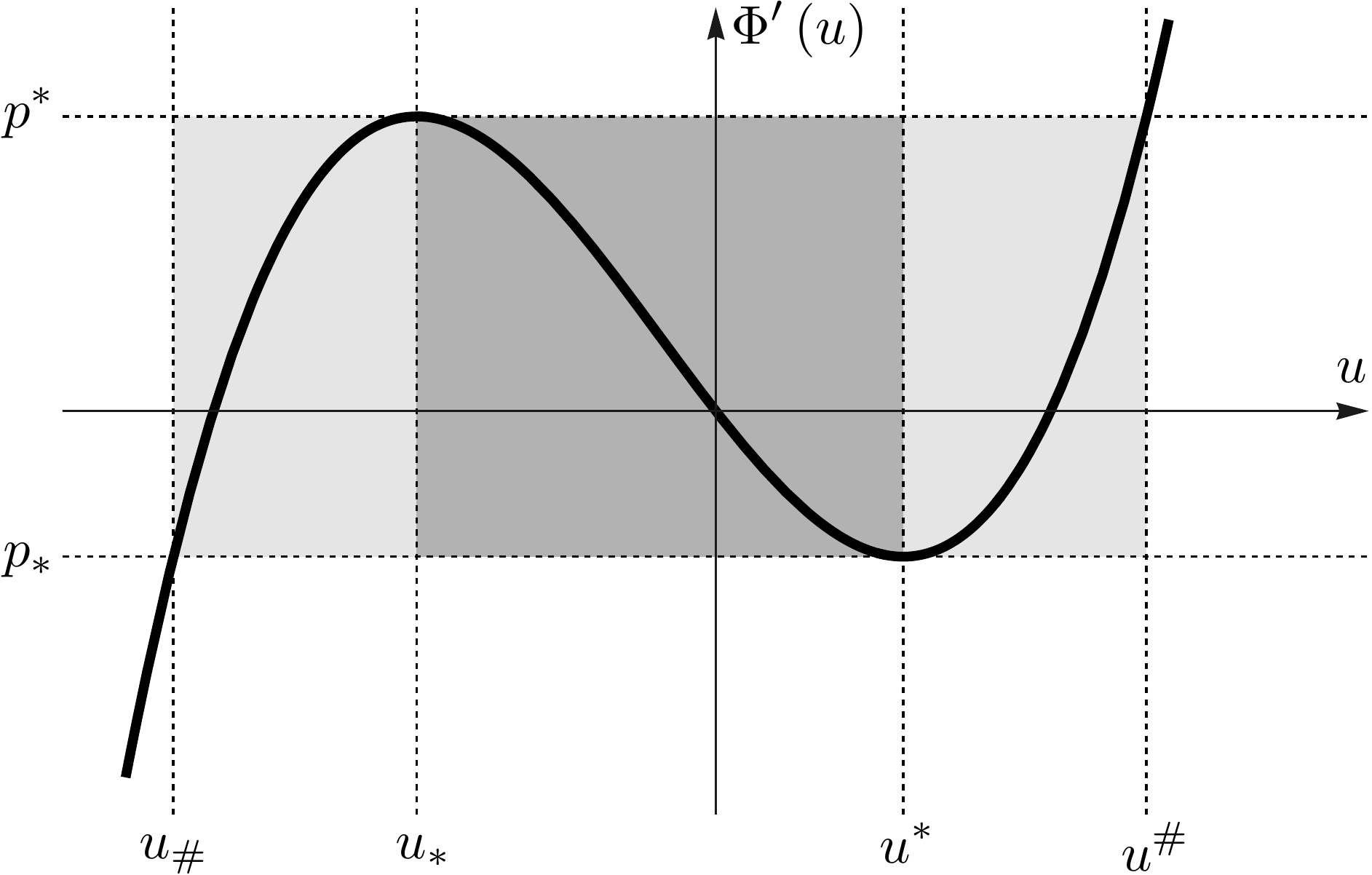}%
\hspace{0.11\textwidth}%
\includegraphics[width=0.4\textwidth]{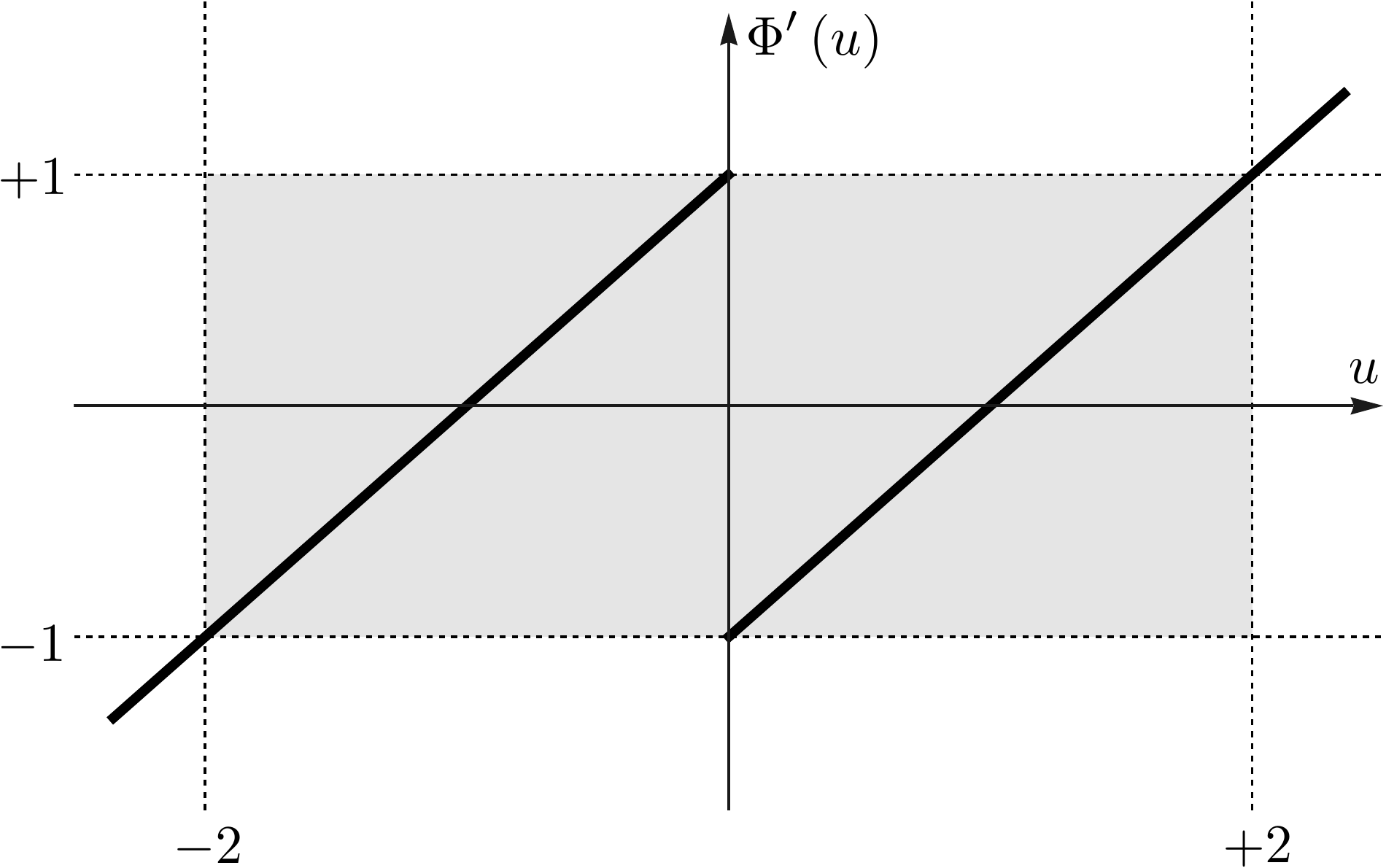}%
}%
\caption[]{\emph{Left panel.} Cartoon of a smooth bistable function $\Phi^\prime$. It possesses an unstable branch on the interval $u\in\ccinterval{u_*}{u^*}$ (spinodal region, dark gray box) as well as the two stable branches for $u\leq u_*$  ($\ominus$-phase) and $u\geq u^*$  ($\oplus$-phase). The light gray rectangles play an important role in the hysteretic flow rule \eqref{eq: hysterese} as they confine the state on either side of the phase interface according to the second part of the Stefan condition \eqref{eq: RH}. \emph{Right panel.} In the bilinear case \eqref{eq: bilinear}, the spinodal region consists only of the origin and the special values can be computed explicitly as in \eqref{eq:BilinConst}.
}
\label{fig: abl_double_well}
\end{figure}
%
%
\paragraph{Hysteretic free boundary problem}
%
%
For any regularization of \eqref{eq: PDE}, the sharp-interface limit combines the bulk diffusion
\begin{align}
\label{eq: bulk}
\partial_t u\pair{t}{x}=\partial_x^2 p\pair{t}{x}\quad\text{for}\quad x\neq \xi_i\at{t}
\end{align}
with the Stefan condition
\begin{align}
\label{eq: RH}
\dot{\xi_i}\at{t}\,\bjump{u\pair{t}{\cdot}}_{x=\xi_i\at{t}}+\bjump{\partial_xp\pair{t}{\cdot}}_{x=\xi_i\at{t}}=0\,,\qquad \bjump{p\pair{t}{\cdot}}_{x=\xi_i\at{t}}=0\,,
\end{align}
which couples the instantaneous interface speed $\dot{\xi}_i\at{t}$ to spatial jumps such as
\begin{align*}
\bjump{u\pair{t}{\cdot}}_{x=\xi_i\at{t}}=\lim_{x\searrow\xi_i\at{t}}u\pair{t}{x}-\lim_{x\nearrow\xi_i\at{t}}u\pair{t}{x}\,.
\end{align*} 
In particular, \eqref{eq: RH} stipulates  that
\begin{align*}
p\pair{t}{x}:=\Phi^\prime\bat{u\pair{t}{x}}
\end{align*}
is continuous at $x=\xi_i\at{t}$ and that the PDE \eqref{eq: bulk} holds even across each interface in a distributional sense. The combination of \eqref{eq: bulk} and \eqref{eq: RH}, however, does not determine the dynamics of the free phase boundaries completely but must be complemented by further equations that take into account the microscopic details in the vicinity of each interface. The extra conditions depend on the particular regularization and must be identified by a careful analysis of the limit $\nu\to0$.  For  instance, the Cahn-Hilliard equation \eqref{eq: CH} fixes the value of $p$ via
\begin{align*}
p\bpair{t}{\xi_i\at{t}}=p_{\text{mw}}\,,
\end{align*}
where the number $p_{\text{mw}}$ is given by the Maxwell construction and depends only on the properties of $\Phi^\prime$, see  for instance \cite{BBMN12,BFG06} concerning more details. For the viscous approximation, however, we expect a more complex behavior. The formal asymptotic analysis in \cite{EP04} predicts the flow rule
\begin{align}
\label{eq: hysterese}
\begin{split}
\dot{\xi_i}\at{t}\,\bjump{u\pair{t}{\cdot}}_{x=\xi_i\at{t}}<0\quad&\Longrightarrow\quad p\bpair{t}{\xi_i\at{t}}=p^*\,,\\
\dot{\xi_i}\at{t}\,\bjump{u\pair{t}{\cdot}}_{x=\xi_i\at{t}}>0\quad&\Longrightarrow\quad p\bpair{t}{\xi_i\at{t}}=p_*\,,\\
p\bpair{t}{\xi_i\at{t}}\in\oointerval{p_*}{p^*}\quad&\Longrightarrow\quad \dot{\xi_i}\at{t}=0\,,
\end{split}
\end{align}
which distinguishes in a hysteretic manner between standing and moving interfaces. Notice that the third equation is actually implied by the first two ones since the second part of the Stefan condition \eqref{eq: RH} ensures in combination with the properties of $\Phi^\prime$ that the interface value of $p$ belongs to the interval $\ccinterval{p_*}{p^*}$. It has also been shown in \cite{EP04} that \eqref{eq: hysterese} is equivalent to the family of entropy inequalities 
\begin{align}
\label{eq: Entropie}
\partial_t \Psi\bat{u\pair{t}{x}}-\partial_x \Bat{\Upsilon\bat{u\pair{t}{x}}\,\partial_x p
\pair{t}{x}}\leq 0
\end{align}
where $\Upsilon$ is an arbitrary but increasing function and $\Psi$ satisfies  $\Psi^\prime\at{u}=\Upsilon\bat{\Phi^\prime\at{u}}$. See \cite{Plo94} for more details concerning the mathematical analysis and \cite{DG17} for the thermodynamical aspects.  We also emphasize that  both the hysteretic flow rule \eqref{eq: hysterese} and the inequality \eqref{eq: Entropie} also arise in  the sharp interface limit $\nu\to0$ of the lattice approximation \eqref{eq: lattice} as shown in \cite{HH13,HH18}.  Moreover, a hysteretic behavior of phase interface can also be observed in other equations such as the nonlocal particle model studied in \cite{MT12,HN23}.
\par
So far, we have no complete understanding of taking the vanishing viscosity limit $\nu\to0$ in the PDE \eqref{eq: visk Reg} although important contributions are given in \cite{EP04,Sma10}. The key problem in a rigorous analysis is to control the strong microscopic fluctuations that are produced by the backward diffusion in the vicinity of a propagating phase interface. In the forthcoming paper \cite{HJ23b} we use a generalized variant of our time-discrete approximation scheme to construct single-interface solutions for $\nu>0$ and to study these fluctuations in greater detail. See also the related traveling wave solutions in \cite{GHJ23}.
%
%
\paragraph{Simplified setting}
%
%
In what follows we restrict our considerations to the  bilinear case 
\begin{align}
\label{eq: bilinear}
\Phi^\prime\at{u}=u-\sgn\at{u},
\end{align}
in which the spinodal region has shrunk to a single point $u=0$. This is illustrated in the right panel of Figure \ref{fig: abl_double_well} and implies
\begin{align}
\label{eq:BilinConst}
u_\#=-2\,,\qquad u_*=0=u^*\,,\qquad u^\#=+2\,,\qquad p_*=-1\,,\qquad p^*=+1\,.
\end{align}
Moreover, we solely study \emph{single-interface solutions} which possess a unique phase interface and satisfy
\begin{align}
\label{eq: single-interface}
\sgn\bat{u\pair{t}{x}}=\sgn\bat{x-\xi\at{t}}\,,
\end{align}
where $\xi\at{t}$ denotes the instantaneous position of the phase interface. Assumptions \eqref{eq: bilinear}+\eqref{eq: single-interface}  simplify the analysis considerably and enable us to represent the solution to a time-discrete analog of the nonlinear equation \eqref{eq: PDE} by means of the linear superposition principle, where the contributions of a propagating phase interface act as localized forcing terms. Notice, however, that the interface curve $\xi$ is not given a priori but must be identified as part of the solution. The bilinear variant of the hysteretic flow rule is illustrated in Figures \ref{fig: hysterese_space} and \ref{fig: hysterese_time}.
\par
%
%
\begin{figure}[ht!]
\medskip%
\centering{
\includegraphics[width=0.85\textwidth]{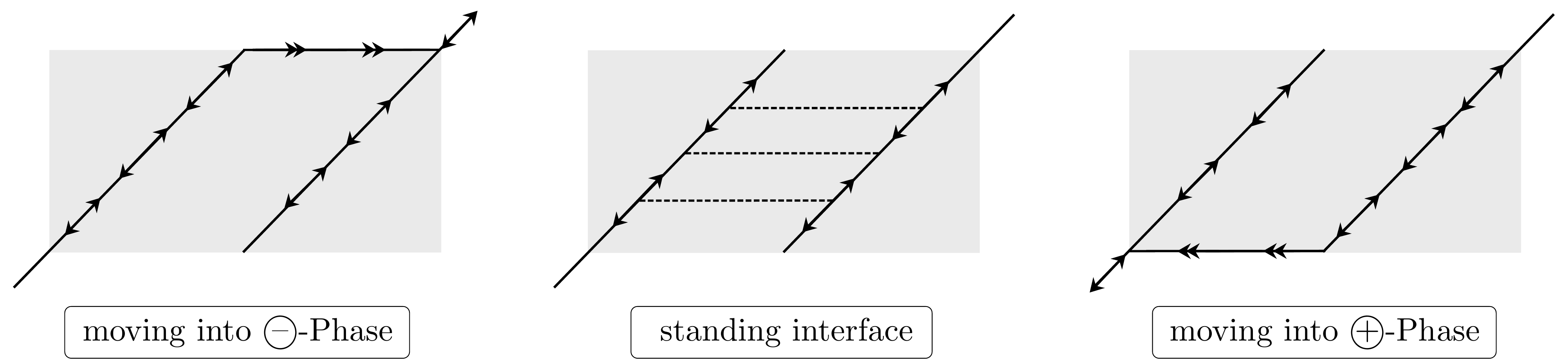}
}
\caption{Illustration of the bilinear analog to the hysteretic flow rule \eqref{eq: hysterese} for a `particle trajectory' \mbox{$t\mapsto \pair{u\pair{t}{x_*}}{p\pair{t}{x_*}}$} with fixed $x_*$ on the graph of \eqref{eq: bilinear} under assumption  \eqref{eq: single-interface}. If the interface runs into the $\ominus$-phase, any \mbox{particle} with $x_*<\xi_\ini$  initially moves along the left branch, suddenly jumps to the $\oplus$-phase, and is \mbox{subsequently} confined to the right branch. An similar discussion covers interfaces that propagate into the $\oplus$-phase but for standing interfaces we have to interpret the corresponding diagram differently. The two particles on either side of the interface are coupled by $\bjump{p\pair{t}{\cdot}}_{x=\xi\at{t}}$ and move along the respective branch of $\Phi^\prime$.
}
\label{fig: hysterese_space}%
\medskip%
\centering{%
\includegraphics[width=0.85\textwidth]{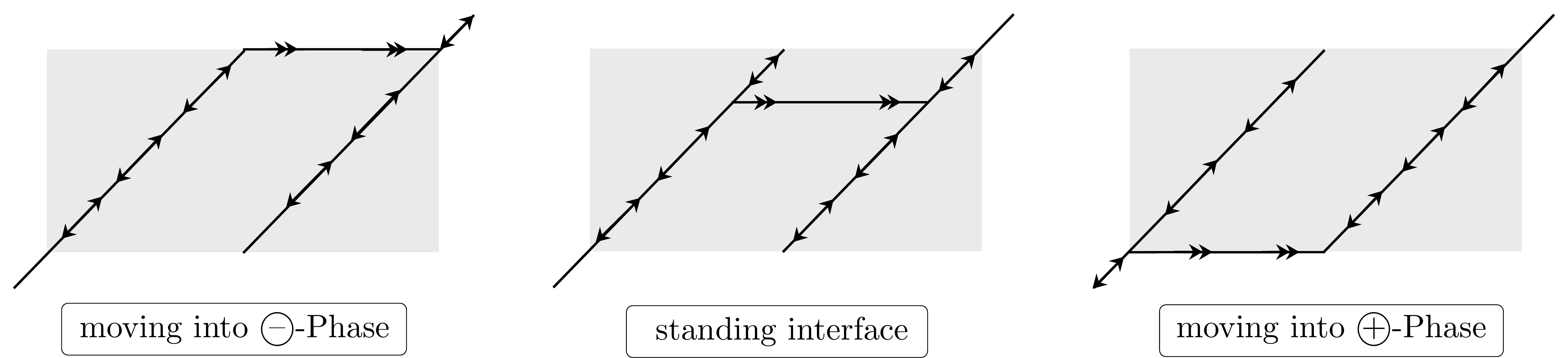}%
}%
\caption{Hysteretic behavior of the spatial dynamics $x\mapsto \pair{u\pair{t_*}{x}}{p\pair{t_*}{x}}$ that corresponds to Figure \ref{fig: hysterese_space}.
}
\label{fig: hysterese_time}
\end{figure}
\%
Piecewise linear functions have already been used in the literature, for instance in \cite{HH13,HH18}, which establish the validity of the free boundary problem \eqref{eq: bulk}$+$\eqref{eq: RH}$+$\eqref{eq: hysterese} in the scaling limit $\nu\to0$ of the lattice regularization \eqref{eq: lattice} for a certain class of single-interface initial data. The lattice results in \cite{GN11,BGN13} allow for  more general nonlinearities and study even the limit $t\to\infty$ but are restricted to initial data that give rise to standing interfaces only.
\par
For trilinear $\Phi^\prime$ and a special class of single-interface data the local existence of a unique solution to the hysteretic limit problem \eqref{eq: bulk}$+$\eqref{eq: RH}$+$\eqref{eq: hysterese} has been shown in \cite{MTT07,MTT09}. Moreover, the global existence has been proven under the assumption $u_\ini\at{x}\in\ccinterval{u_\#}{u^\#}$ which produces standing interfaces but excludes moving ones. These results have been generalized in \cite{Ter11,Sma10,ST13a} and the Riemann problem with both moving and standing interfaces has been solved in \cite{GT10,LM12} for trilinear and more general $\Phi^\prime$. Besides of single-interface solutions one might also study measure-valued solutions to the ill-posed diffusion equation \eqref{eq: PDE}. This generalized concept relies on Young measures and has been introduced in \cite{Plo93,Plo94}. The existence and non-uniqueness of measure-valued solutions has been proven in \cite{Hol83,ST10,ST12,Ter14} while \cite{Ter15} constructs special weak solutions that penetrate the spinodal region.
\par
For smooth bistable functions $\Phi^\prime$, the global well-posedness of the initial value problem follows for the lattice \eqref{eq: lattice} by standard ODE arguments and has been established in \cite{NCP91} for the viscous regularization by transforming \eqref{eq: visk Reg} into an equivalent dynamical system with nonlocal but Lipschitz-continuous right hand side. These results, however, neither cover the discontinuous function \eqref{eq:BilinConst} nor imply the persistence of single-interface data.
%
%
\paragraph{Approximation scheme}
%
%
In this paper we do not study the bistable variant of the free \mbox{boundary} problem  \eqref{eq: bulk}$+$\eqref{eq: RH}$+$\eqref{eq: hysterese} directly but a certain implicit time discretization with constant step size \mbox{$\Delta t=\eps^2$}. In what follows we denote by $u^n$ a single-interface state at time 
\begin{align}
\label{eq:DiscreteTimes}
t^n:=n\,\eps^2
\end{align}
and write $\xi^n$ for the corresponding interface position. Our abstract approximation scheme combines two update rules that provide $u^{n+1}$ and  $\xi^{n+1}$ in dependence  of the current data $\pair{u^n}{\xi^n}$. The latter will be discussed below and concerning the former we observe that the corresponding spatial differential equation 
\begin{align}
\label{eq:ImplicitStep}
\frac{u^{n+1}-u^n}{\eps^2}=\partial_x^2\,\Bat{u^{n+1}-\sgn\bat{\cdot-\xi^{n+1}}}\,,
\end{align}
admits only distributional solutions as it exhibits singular terms on the right hand side. However, the operator $\Id-\eps^2\,\partial_x^2$ is strongly elliptic in sufficiently nice function spaces and its inverse is the convolution with the  fundamental solution
\begin{align}
\label{eq: g}
g_\eps\at{x}:=\frac{1}{2\,\eps}\,\exp\at{-\frac{\abs{x}}{\eps}}\,,
\end{align}
which illustrated in the left panel of Figure \ref{fig: GepsSeps}.
%
%
\begin{figure}[ht!]
\centering{
\includegraphics[width=0.4\textwidth]{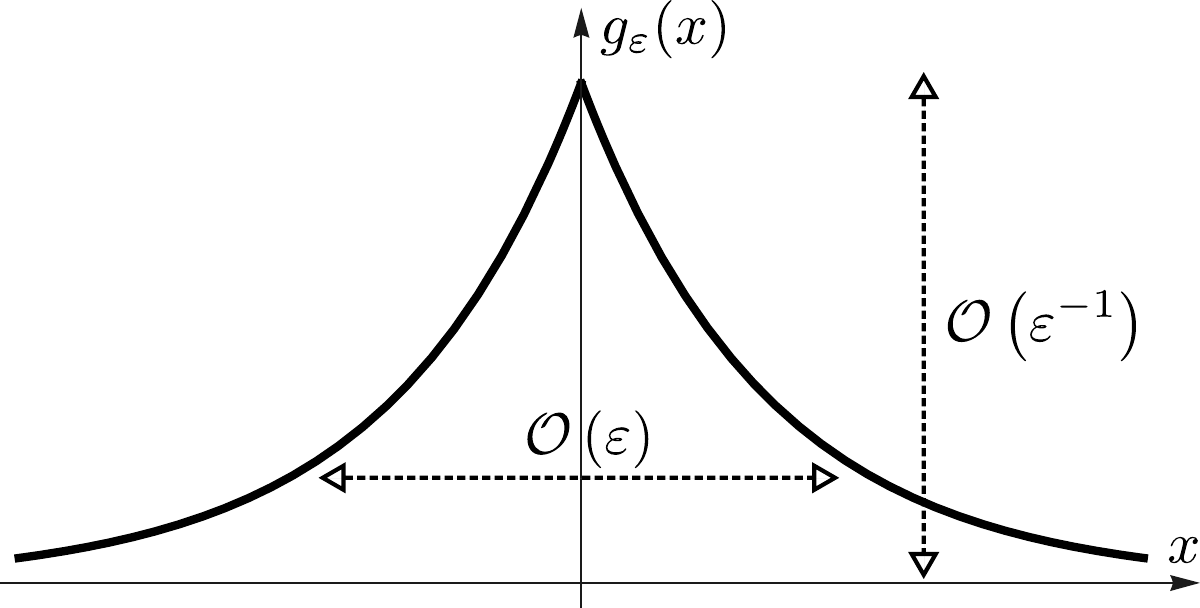}%
\hspace{0.075\textwidth}%
\includegraphics[width=0.4\textwidth]{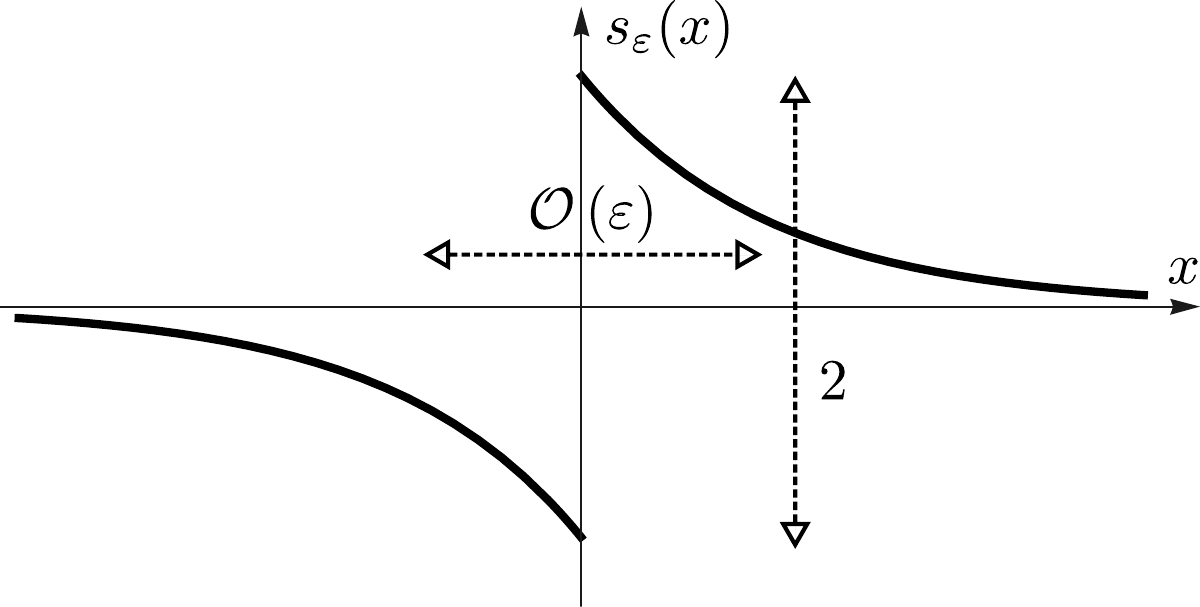}%
}%
\caption{The functions $g_\eps$ and $s_\eps$ from \eqref{eq: g} and \eqref{eq: s} decay exponentially with rate $1/\eps$.}
\label{fig: GepsSeps}
\end{figure}
\%
\par
The scalar function $g_\eps$ satisfies $\bat{\Id-\eps^2\,\partial_x^2}\, g_\eps=\delta_0$, where $\delta_0$ abbreviates the Dirac distribution at $x=0$. In particular,  $\tilde{u}=g_\eps\ast u$ holds if and only if $\tilde{u}-\eps^2\,\partial_x^2 \tilde{u}=u$ and direct computations reveal that the local differential update \eqref{eq:ImplicitStep} is formally equivalent to 
\begin{align}
\label{eq: nu=0, Schema1}
u^{n+1}=\tilde{u}^n+\sgn\at{\cdot-\xi^{n+1}}-\bat{g_\eps\ast\sgn}\at{\cdot-\xi^{n+1}}
\end{align}
with 
\begin{align}
\label{eq: nu=0, Schema1X} \tilde{u}^n:= g_\eps \ast u^n \,.
\end{align}
This non-local equation represents $u^{n+1}$ by integrals of $u^n$ and uses $\xi^{n+1}$ as shift parameter for the function
\begin{align}
\label{eq: s}
\begin{split}
s_\eps\at{x}:=\sgn\at{x}-\bat{g_\eps\ast \sgn}\at{x}=
\begin{cases} 
 \D-\exp\at{+\frac{x}{\varepsilon}}& \text{for } x \leq 0\,,\smallskip\\
\D+ \exp\at{-\frac{x}{\varepsilon}} & \text{for } x >0 \,,
\end{cases}
\end{split}
\end{align}
which exhibits a jump of height $2$ at the origin as shown in the rigth panel of Figure \ref{fig: GepsSeps}. The idea for the update of the interface position is to choose  $\xi^{n+1}$ such that $u^{n+1}$ from \eqref{eq: nu=0, Schema1} is negative and positive for $x<\xi^{n+1}$ and $x>\xi^{n+1}$, respectively. Due to the properties of $s_\eps$ and in reminiscence of the hysteretic flow rule \eqref{eq: hysterese} we distinguish between left-moving, right-moving and standing interfaces in dependence of the value attained by $\tilde{u}^n$ at $\xi^n$.
\begin{scheme}[\bf iteration step $\boldsymbol{n\rightsquigarrow
 n+1}$]\
\label{eq: IE}
\leavevmode
\begin{enumerate}
\item \underline{\emph{convolution of $u^n$:}} 
We compute $\tilde{u}^n$ via  \eqref{eq: nu=0, Schema1X} as convolution of $g_\eps$ and $u^n$.
\item 
\underline{\emph{update of $\xi^n$:}} 
We evaluate $\tilde{u}^n\at{\xi^n}$ and distinct the following three cases 
\begin{align}
\notag
(\mathsf{LM})  \quad \tilde{u}^n\at{\xi^n}>1\,,\qquad (\mathsf{RM})  \quad \tilde{u}^n\at{\xi^n}<-1\,,\qquad (\mathsf{ST})  \quad \tilde{u}^n\at{\xi^n}\in \ccinterval{-1}{1}
\end{align}
as follows.
\begin{align*}
\begin{array}{clc}
(\mathsf{LM})&\text{We choose the maximal value of $\xi^{n+1}$ with $\xi^{n+1}<\xi^{n}$ and $\tilde{u}^n\at{\xi^{n+1}}=+1$.}&\quad\quad\quad
\\%
(\mathsf{RM})&\text{We choose the minimal value of  $\xi^{n+1}$ with $\xi^{n+1}>\xi^{n}$ and $\tilde{u}^n\at{\xi^{n+1}}=-1$.}&
\\%
(\mathsf{ST})&\text{We choose $\xi^{n+1}=\xi^{n}$.}&
\end{array}
\end{align*}
\item\underline{\emph{update of $u^n$:}}
We set $u^{n+1}=\tilde{u}^n+s_\eps\at{\cdot-\xi^{n+1}}$.
\end{enumerate}
\end{scheme}
%
%
%
In Section \ref{sect:scheme} we show that the update rule for $\xi^{n+1}$ is well-defined as long as $u^n$ is a sufficiently nice single-interface state and in Section \ref{sect:convergence} we pass to the continuum limit $\eps\to0$ as outlined below in greater detail. Moreover,  a generalized version is used in \cite{HJ23b} to prove the existence of single-interface solution to the viscous regularization \eqref{eq: visk Reg}.
\par
A hysteretic interface update has also been proposed in \cite{LM12}, which \mbox{studies} several fully \mbox{discretized} schemes for the computation of entropic single-interface solutions to a \mbox{trilinear} variant of \eqref{eq: PDE} and compares the numerical outcome with explicitely known Riemann solutions. One of these schemes also distinguishes explicitely between the three propagation modes for phase interface but the  update rules for both $u$ and $\xi$ are formulated on a spatial lattice and differ from the corresponding formulas in Scheme \ref{eq: IE}. It remains a challenging task to investigate this scheme from \cite{LM12} in greater detail and to prove its convergence in
the limit of vanishing discretization parameters.
%
%
\paragraph{Numerical simulations}
%
To illustrate the key dynamical features of Scheme \ref{eq: IE}, we restrict the space variable $x$ to the interval $\ccinterval{-2}{+2}$, impose homogeneous Neumann conditions at the boundary, set $\eps^2=0.01$, and choose the initial data
\begin{align}
\label{eq: initSim}
u^0\at{x}=
\begin{cases}
2.0\,x-0.6\qquad &\text{for}\;\;\;\;{-}2\leq x\leq 0\,,\\
7.0\,x+1.4\qquad&\text{for}\;\;\;\;0\leq x\leq +2\,,
\end{cases}
\end{align}
which correspond to $\xi^0=0$ and
\begin{align*}
\bjump{u^0}_{x=\xi^0}=2.0\,,\qquad \bjump{\partial_x u^0}_{x=\xi^0}=5.0\,.
\end{align*}
Moreover, instead of implementing the convolutive integrals we solve in each time step  the elliptic auxiliary problem 
\begin{align*}
\tilde{u}^n\at{x}-\eps^2\,\partial^2_x\tilde{u}^n\at{x}=u^n\at{x}\,,\qquad \partial_x\tilde{u}^n\at{\pm2}=0
\end{align*}
by means of a standard finite-difference discretization. The results are shown in Figures \ref{fig:Simulation1}$+$\ref{fig:Simulation2} and illustrate the depinning of an standing interface due to the bulk diffusion. For sufficiently small times $t^n$, the single interface does not move and we have $\xi^n=\xi^0$. The one-sided limits satisfy
\begin{align*}
u^n\at{\xi^0-0}\in\oointerval{-2}{0}\,,\qquad u^n\at{\xi^0+0}\in\oointerval{0}{+2}\,,\qquad
\bjump{u^n}_{x=\xi^0}=u^n\at{\xi^0+0}-u^n\at{\xi^0-0}=2\,.
\end{align*}
and the continuous function
\begin{align}
\label{DefDiscreteP}
p^n:=u^n-\sgn\at{u^n}=u^n-\sgn\at{\cdot-\xi^n}
\end{align}
attains at $x=\xi^n$ a value inside the open interval $\oointerval{-1}{+1}$ as predicted by the hysteretic flow rule \eqref{eq: hysterese}$_3$. We also observe that the value $p^n\at{\xi^n}$ increases in time and suddenly becomes $1$. Starting from this switching time $t\approx0.05$ the phase interface propagates into the $\ominus$-phase and we find
\begin{align*}
p^n\at{\xi^n}=+1\,,\qquad u^n\at{\xi^n-0}=0\,,\qquad u^n\at{\xi^n+0}=2
\end{align*}
in accordance with \eqref{eq: RH}$_2$ and \eqref{eq: hysterese}$_1$. A closer look to the numerical data reveals that both the standing and the left-moving interface also satisfy the discrete analog to the Stefan condition \eqref{eq: RH} up to small error terms. In particular, the initial jump of $\partial_x p^n$ is first smoothed out  but later restored by the moving interface.
%
%
\begin{figure}[ht!]
\centering{
\includegraphics[width=0.3\textwidth]{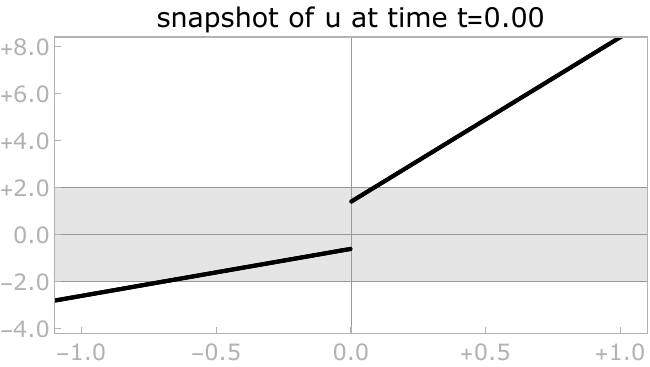}%
\hspace{0.025\textwidth}%
\includegraphics[width=0.3\textwidth]{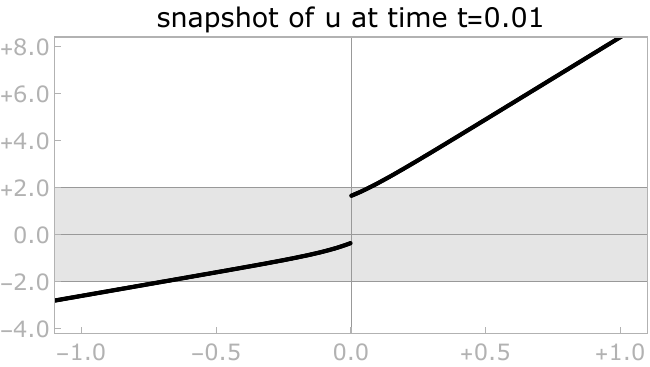}%
\hspace{0.025\textwidth}%
\includegraphics[width=0.3\textwidth]{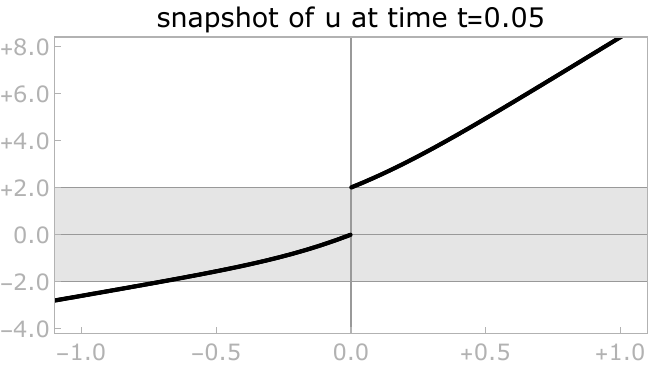}%
\\
\includegraphics[width=0.3\textwidth]{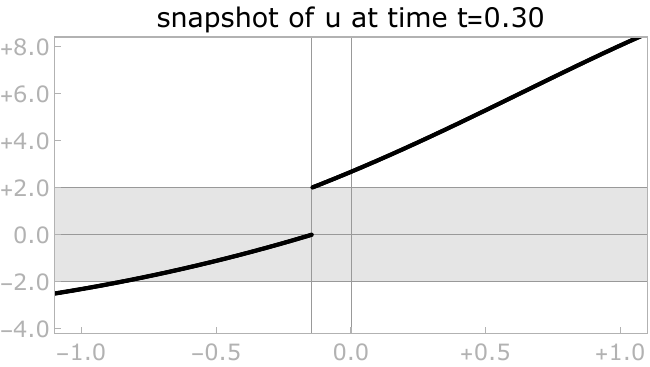}%
\hspace{0.025\textwidth}%
\includegraphics[width=0.3\textwidth]{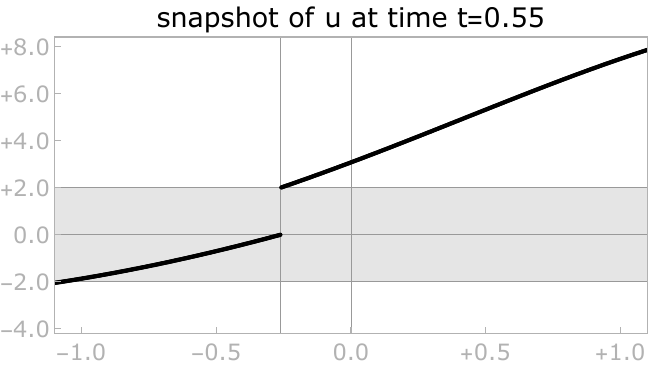}%
\hspace{0.025\textwidth}%
\includegraphics[width=0.3\textwidth]{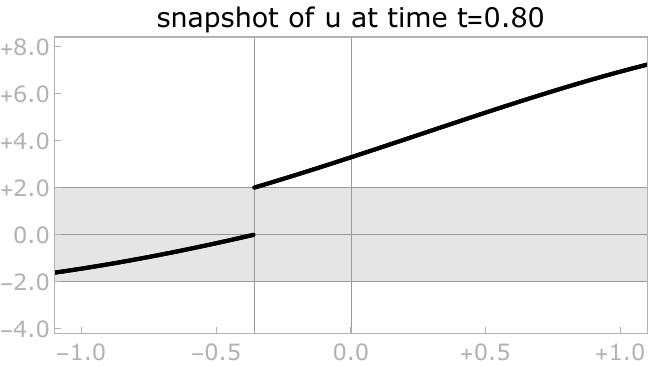}%
}%
\caption{Snapshot of the numerical data  at six non-equidistant times on the subinterval $x\in\ccinterval{-1}{+1}$ of the computational domain $x\in\ccinterval{-2}{+2}$. The initial data are given by \eqref{eq: initSim} and the gray box represents the region $u\in\ccinterval{u_\#}{u^\#}=\ccinterval{-2}{+2}$ as in Figure \ref{fig: abl_double_well}. The vertical lines indicate $\xi^0$ and  $\xi^n$, the initial and the instantaneous interface position, respectively.}
\label{fig:Simulation1}
\end{figure}
%
%
\begin{figure}[ht!]
\centering{
\includegraphics[width=0.3\textwidth]{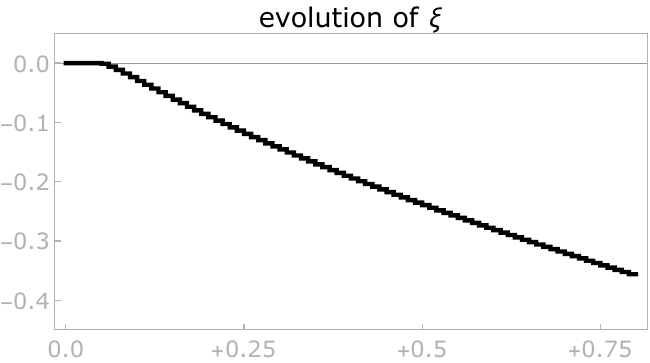}%
\hspace{0.025\textwidth}%
\includegraphics[width=0.3\textwidth]{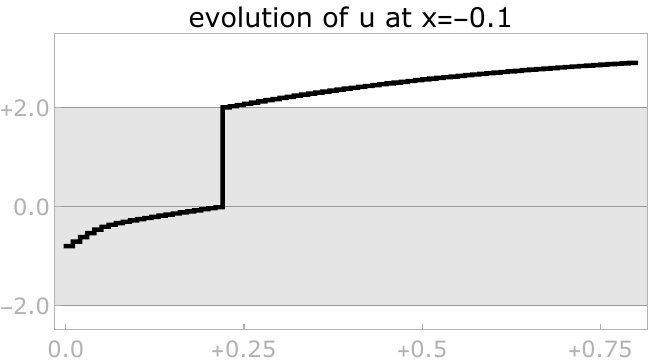}%
\hspace{0.025\textwidth}%
\includegraphics[width=0.3\textwidth]{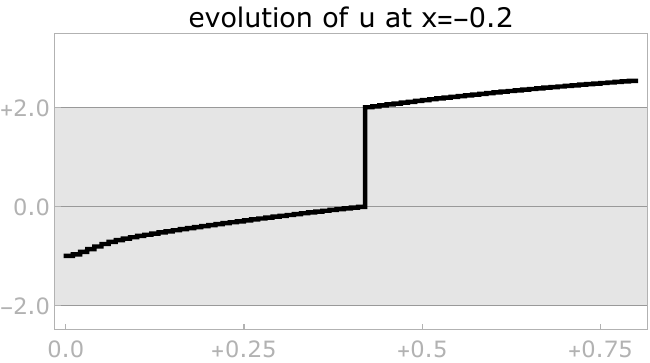}%
}%
\caption{The discrete interface curve and two pointwise trajectories for the numerical data from~Figure \ref{fig:Simulation1}.}
\label{fig:Simulation2}
\end{figure}
%
%
\paragraph{Key findings and overview of the proof strategy}
%
In order to study the temporal continuum limit $\eps\to0$ we interpret the time-discrete data via
\begin{align}
\label{eq:IdentificationX}
p_{\eps}\pair{t}{x}&:=
p^n\at{x}\qquad\text{and}\qquad
\xi^n=\xi_\eps\at{t} \qquad \text{for} \qquad t\in \cointerval{t^n}{t^{n+1}}
\end{align}
as functions that are piecewise constant with respect to $t$. Our main analytical results can informally be summarized as follows.
\begin{result*} 
Scheme \ref{eq: IE} has the following properties.
\begin{enumerate}
\item 
A certain class of single-interface data is invariant and allows for both standing and left-moving interfaces.
\item 
The functions $p^n$ from \eqref{DefDiscreteP} converges as $\eps\to0$ via \eqref{eq:IdentificationX} in a strong sense to a limit function $p_0$ that depends continuously on both $t$ and $x$. Moreover, the limit of the interface positions $\xi^n$ is a non-increasing function $\xi_0$ in the variable $t$. 
\item 
The triple $\triple{u_0}{p_0}{\xi_0}$ with $u_0=p_0+\sgn\at{\cdot-\xi_0}$ is the unique solution to a simplified version of the hysteretic free boundary problem \eqref{eq: bulk}$+$\eqref{eq: RH}$+$\eqref{eq: hysterese}. In particular, the linear bulk diffusion
\begin{align}
\label{eq: simple bulk}
\partial_t p_0\pair{t}{x}=\partial_x^2 p_0\pair{t}{x}\quad \text{for}\quad x\neq \xi_0\at{t}\,,
\end{align}
is satisfied at least in a weak sense and the Stefan condition 
\begin{align}
\label{eq: simple RH}
2\,\dot{\xi}_0\at{t}+\bjump{\partial_x p_0\pair{t}{\cdot}}_{x=\xi_0\at{t}}=0\,,\qquad \bjump{p_0\pair{t}{\cdot}}_{x=\xi_0\at{t}}=0\,,
\end{align}
as well as the hysteretic flow rule
\begin{align}
\label{eq: simple hysteresis}
\begin{split}
\dot{\xi}_0\at{t}<0\;\;\;\Rightarrow\;\;\; p_0\bpair{t}{\xi_0\at{t}}=+1\,,\qquad \quad 
p_0\bpair{t}{\xi_0\at{t}}\in\oointerval{-1}{1}\;\;\;\Rightarrow\;\;\;\,\dot{\xi}_0\at{t}=0.
\end{split}
\end{align}
hold for almost all times $t$. 
\end{enumerate}
\end{result*}
The class of admissible data is defined at the beginning of Section \ref{sect:convergence} and requires in addition to the essential sign conditions a constant lower and a linear upper pointwise bound on the left and the right of the interface, respectively. In Section \ref{sect:persistence} we prove the invariance of these properties under the iteration $n\rightsquigarrow n+1$ and show that they imply the non-strict monotonicity of the interface curve as well as a maximal value  for the modulus of its negative speed. In particular, the interface never propagates into the $\oplus$-phase but can arbitrarily switch between the standing and the left-moving mode. The extra assumptions simplify our asymptotical analysis considerably but might be weakened for the price of more technical and notational effort. In the general case, it might be harder to control the interface speed and one has to divide the time axis into subintervals with either non-increasing or non-decreasing  interface curve so that generalizations of our result can be applied locally. Notice that the general version of the hysteretic flow rule in \eqref{eq: hysterese} predicts a minimal distance between any two times with opposite propagation directions because in between the interface value of $p$ has to change from $p_*=-1$ to $p^*=+1$ or vice versa.
\par
In Section \ref{sect:splitting} we decompose the time-discrete data $p^n$ into a \emph{regular part} which describes the linear diffusion of the initial data and a remaining part which we call \emph{fluctuations} in analogy with the terminology in \cite{HH13,HH18}. A moving interface produces in any time step local fluctuations that have small amplitudes, are initially localized in the vicinity of the phase interface, and spread afterwards diffusively into the bulk. The global fluctuations are the rather huge sum of all local fluctuations --- which are created at different times in distinct places --- and converge in the continuum limit $\eps\to0$ to a continuous function which quantifies how the limit of $p^n$ deviates from a solution of the linear diffusion equation.
\par
In Section \ref{sect:convergence} we pass to the limit $\eps\to0$. We first split the global fluctuation into  negligible and essential ones as described in Section \ref{sect:negligible}. The former are highly irregular but small and vanish in the continuum limit pointwise everywhere. The latter, however, are uniform H\"older continuous with respect to both $t$ and $x$ and hence compact in the space of continuous functions, see Section \ref{sect:essential} for the details.  In Section \ref{sect:essential} we justify the hysteretic free boundary problem along subsequences but this implies also the convergence of the entire family since  it has already been proven that the hysteretic limit model admits a most one solution.
%
%
%
\section{Properties of the scheme}\label{sect:scheme}
%
%
In this section we collect important properties of the time discrete dynamics and start with defining a special class of single-interface data which turns out to be invariant under Scheme \ref{eq: IE}.
\begin{definition}[\bf admissible data]
\label{defi: zul}
We call $\pair{u^n}{\xi^n}$ an admissible single-interface state, if it has the following properties.
\begin{enumerate}
\item \underline{\emph{regularity}}\,: 
The function $u^n$ is continuous for $x<\xi^n$ and $x>\xi^n$.
\item \underline{\emph{jump at the interface}}\,: 
We have $\bjump{u^n}_{x=\xi^n}=2$.
\item  \underline{\emph{refined single-interface property}}\,: 
The pointwise estimates
\begin{align*}
-2\leq u^n\at{x}\leq 0\quad\text{for}\quad x\leq \xi^n\,,\qquad \quad u^n\at{x}\geq 0\quad\text{for}\quad x> \xi^n
\end{align*}
are satisfied.
\item \underline{\emph{majorant property}\,:} 
There exists $\al>0$ such that $u^n\at{x}\leq \alpha\,\at{x-\xi^n}+2$ holds for any $x>\xi^n$.
\end{enumerate} 
\end{definition}
According to the third condition, we only permit data $u^n$ which are greater or equal than $-2$ on the left side of the interface and a easy estimation shows that this implies $\tilde{u}^n\at{\xi^n}\geq -1$. Thus only the cases $(\mathsf{LM})$ and $(\mathsf{ST})$ can occur and the interface is not able to alter its direction of propagation. The majorant property yields an upper bound for the interface speed in Corollary \ref{cor: interfacespeed} and the jump of $u^n$ at the interface ensures the continuity of $p^n$ which is consistent with the second part of the Stefan condition \eqref{eq: RH}. 
\par
Admissible data belong to the function space
\begin{align*}
\fspaceL_{\#}\at{\Rset}:=\Big\{u\in \fspaceL^1_{\loc}\at{\Rset}\;:\;\babs{u\at{x}}\leq a+b\,\babs{x} \text{ for some constants $a$, $b$ and almost all $x\in\Rset$} \Big\}\,.
\end{align*}
and we easily show for any $u\in\fspaceL_{\#}\at\Rset$ that $\tilde{u}=g_\eps\ast u$ is well-defined, likewise an element of  $\fspaceL_{\#}\at\Rset$, and a solution to the differential equation $\tilde{u}\at{x}-\eps^2\,\partial_x^2 \tilde{u}\at{x}=u\at{x}$. Moreover,  the half-space formula
\begin{align}
\label{eq: Halb}
\begin{split}
\tilde{u}\at{x}&=\frac{1}{\eps}\,\sinh\at{\frac{x-\zeta}{\eps}}\,\int\limits_{x}^{\infty}\exp\at{-\frac{y-\zeta}{\eps}}\,u\at{y}\dint{y}
\\&%
\quad+\frac{1}{\eps}\,\exp\at{-\frac{x-\zeta}{\eps}}\,\int\limits_{\zeta}^{x}\sinh\at{\frac{y-\zeta}{\eps}}\,u\at{y}\dint{y}+\tilde{u}\at{\zeta}\,\exp\at{-\frac{x-\zeta}{\eps}}\,,
\end{split}
\end{align}
holds with $\zeta\leq x <\infty$ for any $\zeta\in\Rset$ and represents $\tilde{u}$ on each interval $\oointerval{\zeta}{\infty}$ in terms of the respective boundary value $\tilde{u}\at{\zeta}=\lim_{x\searrow \zeta} \tilde{u}\at{x}$. Below we employ  \eqref{eq: Halb} with $u=u^n$ and either $\zeta=\xi^{n+1}$ or $\zeta=\xi^n$ as well as corresponding formula for $x\leq \zeta$ that follows by setting $x\mapsto 2\,\zeta -x$.
%
%
\subsection{Existence of time-discrete single-interface solutions}
\label{sect:persistence}
%
%
In the subsequent two lemmas we prove that admissible single-interface data are invariant under Scheme \ref{eq: IE}.
\begin{lemma}[\bf persistence of single-interface property]
\label{lem: single}
Suppose $\pair{u^n}{\xi^n}$ is admissible in the sense of Definition \ref{defi: zul}. Then the update step in Scheme \ref{eq: IE} is well-defined and $\pair{u^{n+1}}{\xi^{n+1}}$ fulfills the refined single-interface property.
\end{lemma}
\begin{proof}
\underline{\emph{Preliminaries}}\,: 
Since $u^n\in L_\#\at{\Rset}$, the convolution with the exponentially decaying kernel $g_\eps$ from \eqref{eq: g} is well-defined. In case of $-1\leq\tilde{u}^n\at{\xi^n}\leq 1$, the new interface position exists trivially via $\xi^{n+1}=\xi^n$. Otherwise we have  $\tilde{u}^n\at{\xi^n}>1$ and inserting the single-interface property as well as  the majorant for $u^n$ in the convolution formula we get
\begin{align*}
\tilde{u}^n\at{x}\leq\frac{1}{2\,\eps}\int\limits_{\xi^n}^{+\infty}\exp\at{\frac{x-y}{\eps}}\,\bat{\alpha\,\at{y-\xi^n}+2}\dint{y}=\frac{1}{2}\,\at{\alpha+2\,\eps}\,\exp\at{\frac{x-\xi^n}{\eps}}
\end{align*}
for all $x<\xi^n$. Since the right hand side converges to $0$ as $x\to-\infty$, both the existence and uniqueness of $\xi^{n+1}<\xi^n$ follow from the continuity of $\tilde{u}^n$ and due to the maximality condition in Scheme \ref{eq: IE}. To establish the single-interface property of $u^{n+1}$ , we distinguish the following three cases.
\par\underline{\emph{Case $x\le\xi^{n+1}$}}\,: 
For a moving interface, we insert $u^n\at{x}\leq 0$ as well as $\tilde{u}^n\at{\xi^{n+1}}=1$ in the reflected variant of the half-space-formula \eqref{eq: Halb} with $u=u^n$, $\zeta=\xi^{n+1}$. This gives in combination with \eqref{eq: s} the estimate
\begin{align*}
u^{n+1}(x)=\tilde{u}^n(x)+s_\eps\at{x-\xi^{n+1}}\le\exp\left(\frac{x-\xi^{n+1}}{\eps}\right)-\exp\left(\frac{x-\xi^{n+1}}{\eps}\right)=0\,,
\end{align*}
which also holds for a standing interface due to $\tilde{u}^n\at{\xi^{n}}\leq 1$. Thanks to $u^n\at{x}\geq -2$ we also get
\begin{align*}
u^{n+1}\at{x}&=\tilde{u}^n\at{x}-\exp\at{\frac{x-\xi^{n+1}}{\nu}}\geq -2+\exp\at{\frac{x-\xi^{n+1}}{\nu}}-\exp\at{\frac{x-\xi^{n+1}}{\nu}}=-2
\end{align*}
for both moving and standing interfaces. 
\par\underline{\emph{Case $\xi^{n+1}< x\le\xi^n$}}\,: 
This case is only relevant for a moving interface and $\tilde{u}^n\at{x}\ge 1$ implies
\begin{align*}
u^{n+1}\at{x}\geq 1+\exp\at{-\frac{x-\xi^{n+1}}{\eps}}\geq 1\,.
\end{align*}
\par
\underline{\emph{Case $x>\xi^n$}}\,: 
For a moving interface, the half-space formula \eqref{eq: Halb} with $u=u^n$, $\zeta=\xi^n$ and the estimate $\tilde{u}^n\at{\xi^n}\geq 1$ imply
\begin{align*}
u^{n+1}\at{x}\geq \exp\at{-\frac{x-\xi^n}{\eps}}+\exp\at{-\frac{x-\xi^{n+1}}{\eps}}\geq 0\,.
\end{align*}
On the other hand, for a standing interface we get
\begin{align*}
u^{n+1}\at{x}\geq - \exp\at{-\frac{x-\xi^n}{\eps}}+\exp\at{-\frac{x-\xi^{n}}{\eps}}=0
\end{align*}
thanks to $\tilde{u}^n\at{\xi^n}\geq -1$.
\end{proof}
\begin{lemma}[\bf majorant property]
\label{lem: Maj xi}
The pair $\pair{u^{n+1}}{\xi^{n+1}}$ from Lemma \ref{lem: single} satisfies
\begin{align*}
u^{n+1}\at{x}\leq 2+\alpha\,\at{x-\xi^{n+1}}\,.
\end{align*}
for $x > \xi^{n+1}$.
\end{lemma}
\begin{proof}
\underline{\emph{Moving interface, $\xi^{n+1}< x\leq \xi^n$}}\,: 
We insert the pointwise upper bounds for  $u^n$ into the half-space formula \eqref{eq: Halb} with $u=u^n$, $\zeta=\xi^{n+1}$ and use also $\tilde{u}^{n}\at{\xi^{n+1}}=1$. This gives
\begin{align*}
u^{n+1}\at{x}
&\leq \frac{1}{\eps}\,\sinh\at{\frac{x-\xi^{n+1}}{\eps}}\,\int\limits_{\xi^n}^{\infty} \exp\at{-\frac{y-\xi^{n+1}}{\eps}}\,\bat{2+\al\at{y-\xi^n}}\dint{y}+0+2\,\exp\at{-\frac{x-\xi^{n+1}}{\eps}}\\
&=\sinh\at{\frac{x-\xi^{n+1}}{\eps}}\, \exp\at{-\frac{\xi^n-\xi^{n+1}}{\eps}}\,\at{2+\alpha\,\eps}+2\,\exp\at{-\frac{x-\xi^{n+1}}{\eps}}\\
&\leq \exp\at{-\frac{x-\xi^{n+1}}{\eps}}\,\at{\sinh\at{\frac{x-\xi^{n+1}}{\eps}}\,\at{2+\alpha\,\eps}+2}\,,
\end{align*}
where we computed the integral explicitly and combined the monotonicity of $\exp$ with $-\xi^n\leq -x$. Exploiting the elementary estimates
\begin{align*}
\exp\at{-\eta}\,\bat{2\,\sinh\at{\eta}+2}\leq 2\,,\qquad\exp\at{-\eta}\,\sinh\at{\eta}\leq \eta\qquad\text{for all}\qquad \eta\in\Rset
\end{align*}
with $\eta=\at{x-\xi^{n+1}}/\eps$ we finally obtain 
\begin{align*}
u^{n+1}\at{x}\leq 2+\al\,\eps\,\eta=2+\al\,\at{x-\xi^{n+1}}
\end{align*}
and hence the desired estimate on a a first subinterval.
\par\underline{\emph{Moving interface, $x\geq\xi^{n}$}}\,: 
This time we use \eqref{eq: Halb} with $u=u^n$, $\zeta=\xi^n$ and compute
\begin{align}
\label{eq: aux1}
\begin{split}
u^{n+1}\at{x}
&\leq \frac{1}{\eps}\,\sinh\at{\frac{x-\xi^{n}}{\eps}}\,\int\limits_{x}^{\infty} \exp\at{-\frac{y-\xi^{n}}{\eps}}\,\bat{2+\al\at{y-\xi^n}}\dint{y}\\
&\quad+ \frac{1}{\eps}\,\exp\at{-\frac{x-\xi^{n}}{\eps}}\,\int\limits_{\xi^n}^{x} \sinh\at{\frac{y-\xi^{n}}{\eps}}\,\bat{2+\al\at{y-\xi^n}}\dint{y}\\
&\quad+\tilde{u}^n\at{\xi^n}\,\exp\at{-\frac{x-\xi^{n}}{\eps}}+\exp\at{-\frac{\xi^n-\xi^{n+1}}{\eps}}\,\exp\at{-\frac{x-\xi^{n}}{\eps}}\,.
\end{split}
\end{align}
Since $u^{n+1}\at{\xi^n}=\tilde{u}^n\at{\xi^n}+\exp\at{-\frac{\xi^n-\xi^{n+1}}{\eps}}$ holds due to Scheme \ref{eq: IE}, we obtain
\begin{align*}
u^{n+1}\at{x}&\leq \at{2+\alpha\,\at{x-\xi^n}-2\,\exp\at{-\frac{x-\xi^n}{\eps}}}+u^{n+1}\at{\xi^n}\,\exp\at{-\frac{x-\xi^{n}}{\eps}}
\end{align*}
and the first part of the proof ensures $u^{n+1}\at{\xi^n}\leq 2+\al\,\at{\xi^n-\xi^{n+1}}$. In combination we get
\begin{align*}
u^{n+1}\at{x}&\leq
2+\alpha\,\at{x-\xi^n}+\alpha\,\at{\xi^n-\xi^{n+1}}\leq 2+\alpha\,\at{x-\xi^{n+1}}
\end{align*}
thanks to $\exp\bat{-\at{x-\xi^n}/\eps}\leq 1$.
\par\underline{\emph{Standing interface, $x>\xi^{n}=\xi^{n+1}$}}\,: 
In this case, formula \eqref{eq: aux1} ensures
\begin{align*}
u^{n+1}\at{x}&\leq \at{2+\alpha\,\at{x-\xi^n}-2\,\exp\at{-\frac{x-\xi^n}{\eps}}}+\tilde{u}^{n}\at{\xi^n}\,\exp\at{-\frac{x-\xi^{n}}{\eps}}+\exp\at{-\frac{x-\xi^n}{\eps}}
\end{align*}
and the claim follows since $\tilde{u}^n\at{\xi^n}\leq 1$ holds by construction.
\end{proof}
Lemma \ref{lem: single} and \ref{lem: Maj xi} imply for any admissible $\bpair{u^{n}}{\xi^{n}}$ that $\bpair{u^{n+1}}{\xi^{n+1}}$ is also admissible in the sense of Definition \ref{defi: zul} with the same value of $\al$ since the piecewise continuity of $u^{n+1}$ and the corresponding jump condition at $x=\xi^{n+1}$ hold by construction.  For our analysis in Section \ref{sect:convergence} we further need an upper bound for the discrete interface speed.
\begin{corollary}[\bf interface speed]
\label{cor: interfacespeed}
Suppose $\pair{u^n}{\xi^n}$ is a single-interface-solution. Then we have
\begin{align*}
0\leq\xi^{n}-\xi^{n+1}\le \frac{\alpha}{2}\,\eps^2\,.
\end{align*}
\end{corollary}
\begin{proof}
Since $\pair{u^n}{\xi^n}$ is a single-interface-solution, the function $\bar{u}^n$ with
\begin{align*}
\bar{u}^n(x)=
\begin{cases} 
 0       & \text{for } x \le \xi^n\,, \\
    \alpha\,\at{x-\xi^n}+2 & \text{for } x > \xi^n 
\end{cases}
\end{align*}
is a majorant of $u^n$ by Lemma \ref{lem: Maj xi}. We conclude that
\begin{align*} 
\bar{\tilde{u}}^n(x)=
\begin{cases} 
 \frac{1}{2}\exp\left(\frac{x-\xi^n}{\eps}\right)(\alpha\,\eps+2)       & \text{for } x \le \xi^n\,,  \\
  \alpha\,\at{x-\xi^n}+\frac{1}{2}\exp\left(-\frac{x-\xi^n}{\eps}\right)(\alpha\,\eps-2)+2 & \text{for } x > \xi^n  
\end{cases}
\end{align*}
is majorant for $\tilde{u}^{n}$ and obtain
\begin{align*}
\xi^n-\bar{\xi}^{n+1}\ge\xi^n-\xi^{n+1}\geq 0\,,
\end{align*}
where the auxiliary position $\bar{\xi}^{n+1}$ is defined by $\bar{\tilde{u}}^n(\bar{\xi}^{n+1})=1$ and satisfies
\begin{align*}
\xi^n-\bar{\xi}^{n+1}=\eps\ln\at{1+\frac12\,\alpha\,\eps}\le\frac{\alpha}{2}\,\eps^2\,.
\end{align*}
The claim now follows immediately.
\end{proof}
%
%
\subsection{Splitting of the data and introdution of the fluctuations}
\label{sect:splitting}
%
%
For the analytical considerations in the next section it is convenient to replace $u^n$ by
\begin{align*}
p^n:=u^n-\sgn\at{u^n}=u^n-\sgn\at{\cdot-\xi^n}
\end{align*}
since this function is continuous at $x=\xi^n$. The dynamics of $u^n$ and the single-interface property imply
\begin{align}
\label{sc: local p1}
\begin{split}
p^{n+1}&=u^{n+1}-\sgn\at{\cdot-\xi^{n+1}}\\&
=g_{\eps}\ast \bat{p^{n}+\sgn\at{\cdot-\xi^n}}+s_\eps\at{\cdot-\xi^{n+1}}-\sgn\at{\cdot-\xi^{n+1}}
\\&=g_{\eps}\ast p^{n}-r^n
\end{split}
\end{align}
thanks to \eqref{eq: nu=0, Schema1} and \eqref{eq: s}. Here, 
the \emph{local fluctuations} are given by
\begin{align}
\label{sc: local p2}
\begin{split}
r^{n}&:= -g_{\eps}\ast\sgn\at{\cdot-\xi^{n}}-s_\eps\at{\cdot-\xi^{n+1}}+\sgn\at{\cdot-\xi^{n+1}}\\
&=g_{\eps}\ast\sgn\at{\cdot-\xi^{n+1}}-g_{\eps}\ast\sgn\at{\cdot-\xi^{n}}
\end{split}
\end{align}
and represent the contributions that stem from the shift of the interface in the $n^{\text{th}}$ time step from $\xi^n$ to the updated position $\xi^{n+1}\leq\xi^n$. We emphasize that $r^n$ is continuous, strictly positive, and strongly localized with
\begin{align}
\label{eq:FormR}
r^{n}\at{x}=
\begin{cases}
\exp\at{\frac{x-\xi^{n+1}}{\eps}}-\exp\at{\frac{x-\xi^{n}}{\eps}} & \text{for }x\leq\xi^{n+1}\,,\\
\exp\at{-\frac{x-\xi^{n}}{\eps}}-\exp\at{-\frac{x-\xi^{n+1}}{\eps}} & \text{for }x\geq\xi^{n}\,,\\
2-\exp\at{\frac{x-\xi^{n}}{\eps}}-\exp\at{-\frac{x-\xi^{n+1}}{\eps}} & \text{for }\xi^{n+1}\leq x\leq\xi^{n}\,,
\end{cases}
\end{align}
see Figure \ref{fig: locFluc} for an illustration. The discrete Duhamel principle combined with the initial conditions
\begin{align*}
q^0=p^0=u^0-\sgn\at{\cdot-\xi^0}\,,\qquad f^0=0
\end{align*}
implies the representation formula
\begin{align}
\label{eq: iE}
p^{n}=\underbrace{\at{g_{\eps}\ast\ldots\ast g_{\eps}}}_{\substack{\text{$n$ times}}}\ast p^0-f^{n}\,.
\end{align}
The \emph{global fluctuations} are defined by 
\begin{align}
\label{eq:DefF}
f^{n}\at{x}&:=r^{n-1}\at{x}+g_{\eps}\ast r^{n-2}\at{x}+\ldots+\underbrace{\at{g_{\eps}\ast\ldots\ast g_{\eps}}}_{\substack{\text{$n{-}1$ times}}}\ast\, r^0\at{x}=\sum_{i=1}^{n} \underbrace{\at{g_{\eps}\ast\ldots\ast g_{\eps}}}_{\substack{\text{$n{-}i$ times}}}\ast \,r^{i-1}\at{x}
\end{align}
for $n\geq1$ and collect the terms from all interface jumps in the past. Without fluctuations, equation \eqref{sc: local p1} can be regarded as an \mbox{implicit} \mbox{scheme} for the linear diffusion equation and defining the \emph{regular part} via
\begin{align}
\label{sc: q}
q^{n+1}:=g_\eps\ast q^{n}\quad\text{with}\quad q^0=p^0
\end{align}
we obtain 
\begin{align*}
p^n=q^n-f^n\,.
\end{align*}
This decomposition is central for our asymptotic analysis and splits  $p^n$ into a term that depends diffusively on the initial data only and a contribution of the phase interface that exhibits much less regularity. Notice, however, that  $\xi^n$ is not known a priori but depends in a subtle and nonlinear way on the regular part $q^{n-1}$ as well as on all previous interface positions. In what follows we argue that the regular part approaches a solution of the linear diffusion equation in the limit $\eps\to0$. The convergence of the global fluctuations, however, is much more involved and requieres a better understanding of the fine strucure on different spatial scales. In particular, in Section \ref{sect:convergence} we split $f^n$ further into many negligible and one essential part. The former are very irregular but small while the latter turns out to be uniformly bounded in some H\"older space, see Figure \ref{fig: strategy} for a schematic overview.
\par
In a preparatory step we compare the local fluctuations $r^n$ from \eqref{sc: local p2} by a suitable scaled and shifted version of the convolution kernel $g_\eps$. More precisely, we set
\begin{align}
\label{eq:DefREss}
\begin{split}
r^{n}_{\ess}\at{x}&:=2\,\eps^2\,\frac{\xi^{n}-\xi^{n+1}}{\eps^2}\,g_{\eps}\at{x-\frac{1}{2}\at{\xi^{n+1}+\xi^{n}}}\\
&=
\begin{cases}\D
\frac{\xi^{n}-\xi^{n+1}}{\eps}\,\exp\at{\frac{x-\frac{1}{2}\at{\xi^{n+1}+\xi^{n}}}{\eps}} & \text{for } x\leq\frac{1}{2}\at{\xi^{n+1}+\xi^{n}}\,,\\\D
\frac{\xi^{n}-\xi^{n+1}}{\eps}\,\exp\at{-\frac{x-\frac{1}{2}\at{\xi^{n+1}+\xi^{n}}}{\eps}} & \text{for } x\geq\frac{1}{2}\at{\xi^{n+1}+\xi^{n}}
\end{cases}
\end{split}
\end{align}
and show that the difference to the local fluctuations is small.
%
%
\begin{figure}[ht!]
\centering{
\includegraphics[width=0.4\textwidth]{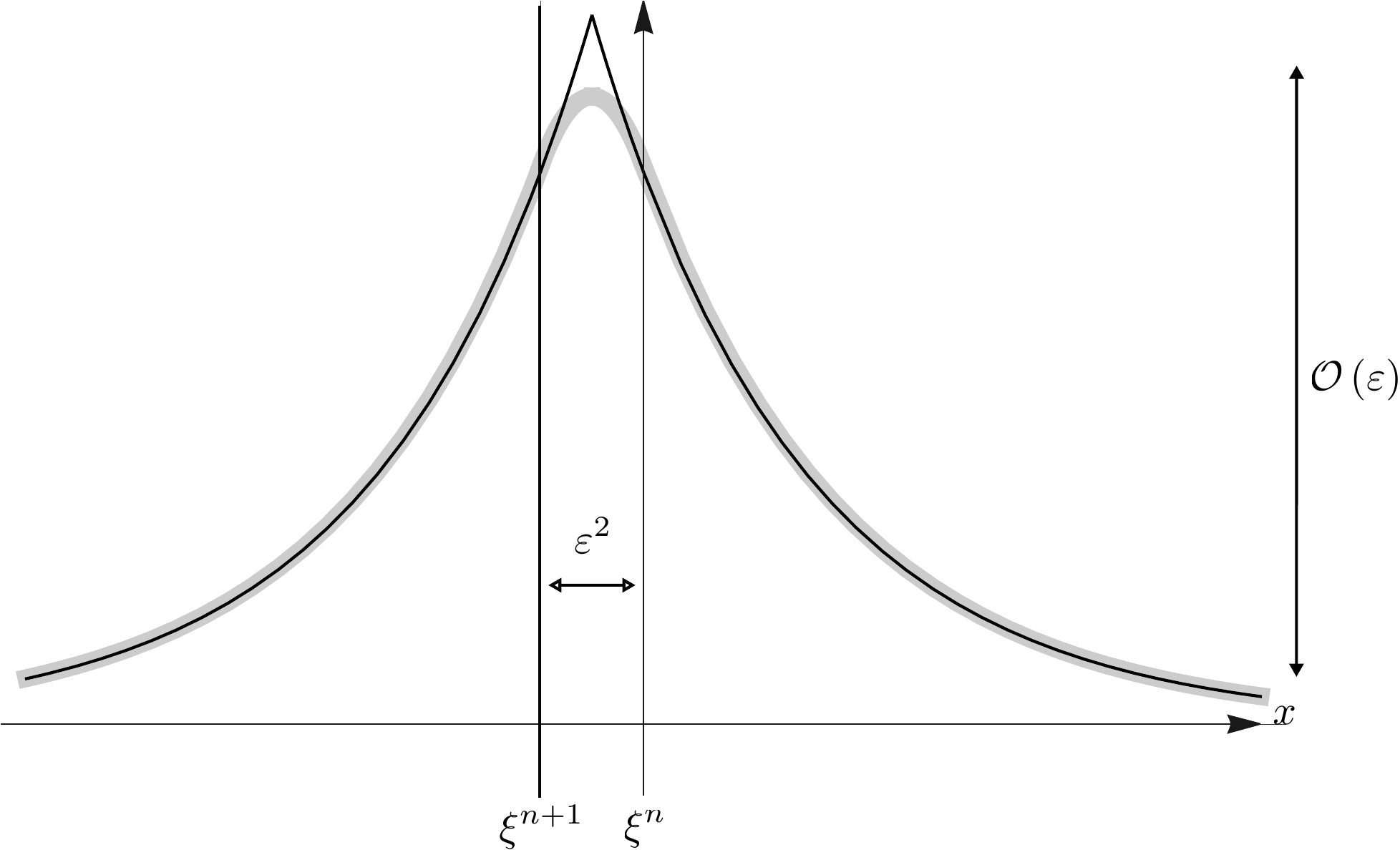}
}%
\caption{Cartoon of the local fluctuations $r^n$ (grey) and $r^n_{\ess}$ (black).}
\label{fig: locFluc}
\end{figure}
%
%
\begin{lemma}[\bf approximation error for local fluctuations]
\label{lem: eff}
We have
\begin{align*}
\babs{r^{n}\at{x}-r_{\ess}^{n}\at{x}}\leq C\,\al\,\eps\,\babs{r^{n}_{\ess}\at{x}}
\end{align*}
for all $x\in\Rset$ and $n\leq N$.
\end{lemma}
\begin{proof}
Lemma \ref{cor: interfacespeed} ensures
\begin{align*}
\mu_n:=\frac{\xi^{n}-\xi^{n+1}}{\eps}=\mathcal{O}\at{\al\,\eps}
\end{align*}
and introducing the spatial variable $y$ by
\begin{align*}
x=\frac{\xi^n+\xi^{n+1}}{2}+\eps\,\mu_n\,y
\end{align*}
we find that $x=\xi^{n+1}$ and $x=\xi^n$ corresponds to $y=-1/2$ and $y=+1/2$, respectively.
\par\underline{\emph{Case  $\abs{y}\leq \frac12$}}\,: 
Using Taylor expansion in \eqref{eq:FormR} and \eqref{eq:DefREss} we get
\begin{align*}
r^n\at{x}&=2-\exp\at{\mu_n\,y-\frac12\,\mu_n}-\exp\at{-\mu_n\,y-\frac12\,\mu_n}=\mu_n+\mathcal{O}\at{\mu_i^2}
\end{align*}
as well as
\begin{align*}
r^n_{\ess}\at{x}=\mu_n\,\exp\bat{-\mu_n\,\abs{y}}=\mu_n+\mathcal{O}\at{\mu_n^2}
\end{align*}
and this implies via
\begin{align*}
\abs{\frac{r^n\at{x}-r^n_{\ess}\at{x}}{r^n_{\ess}\at{x}}}=\mathcal{O}\at{\mu_n}
\end{align*}
the desired estimate inside the interval $\ccinterval{\xi^{n+1}}{\xi^n}$.
\par\underline{\emph{Case $\abs{y}\geq \frac12$}}\,: 
We compute
\begin{align*}
r^n\at{x}=\exp\at{-\mu_n\,\abs{y}}\,\at{\exp\at{\frac12\,\mu_n}-\exp\at{-\frac12\,\mu_n}}\,, \qquad\ r^n_{\ess}\at{x}=\mu_n\,\exp\at{-\mu_n\,\abs{y}}
\end{align*}
and obtain
\begin{align*}
\abs{\frac{r^n\at{x}-r^n_{\ess}\at{x}}{r^n_{\ess}\at{x}}}=\abs{\frac{\exp\at{\frac12\,\mu_n}-\exp\at{-\frac12\,\mu_n}-\mu_n}{\mu_n}}=\frac{\mathcal{O}\at{\mu_n^3}}{\mu_n}=\mathcal{O}\at{\mu_n^2}
\end{align*}
by using a direct Taylor argument.
\end{proof}
%
%
%
\section{Compactness and convergence to the limit model}\label{sect:convergence}
%
%
In this section we interpret the data provided by Scheme \ref{eq: IE} as functions that are piecewise constant with respect to the continuum time $t\in\ccinterval{0}{T}$ and depend additionally on the space variable $x\in\Rset$. More precisely, we rely on the interpolation formula \eqref{eq:IdentificationX}, which involves the discrete times from \eqref{eq:DiscreteTimes}, and define $u_\eps$, $q_\eps$, $f_\eps$ as piecewise constants counterparts to  $u^n$, $q^n$, $f^n$ in a similar way. We further fix a macroscopic final time $0<T<\infty$ independent of $\eps$ and denote by $N$ the corresponding time index with
\begin{align*}
N\,\eps^2\leq T<\at{N+1}\,\eps^2\,.
\end{align*}
The discrete time index $n$ and the continuous time variable $t$ in formula \eqref{eq:IdentificationX} are coupled by the relation
\begin{align}
\label{eq: tn couple}
t\in\cointerval{n\,\eps^2}{\at{n+1}\,\eps^2}\subseteq\ccinterval{0}{T}\,,
\end{align}
but sometimes we also exploit the alternative representations
\begin{align*}
p_{\eps}\pair{t}{x}=\sum_{m=0}^{N}\xi^m\at{x}\,\chi_{\cointerval{m\,\eps^2}{\at{m+1}\,\eps^2}}\at{t}\,,\qquad 
\xi_{\eps}\at{t}=\sum_{m=0}^{N}\xi^m\,\chi_{\cointerval{m\,\eps^2}{\at{m+1}\,\eps^2}}\at{t}\,,
\end{align*}
where $\chi_I$ denotes as usual the characteristic function of the interval $I$. In order to establish uniform estimates for the limit $\eps\to0$ we always assume 
\begin{align}
\label{eq:IdentificationY}
 \eps^2\leq T\,,\qquad 0\leq t\leq T\,,\qquad 0\leq n\leq N+1\,,
\end{align}
which ensures that all terms below are well-defined, and denote by $C$ any generic constant that is independent of $\eps$, $T$, and the initial data. The latter comply with the following requirements.
\begin{assumption}[\bf admissible initial data]
\label{Ann: initial data}
The initial data $\pair{u_{\ini}}{\xi_{\ini}}$ have the following properties:
\begin{enumerate}
\item $\pair{u_{\ini}}{\xi_{\ini}}$ is admissible in the sense of Definition \ref{defi: zul},
\item $p_{\ini}=u_{\ini}-\sgn\at{u_\ini}$ grows linearly at infinity in the sense that $p_{\ini}^{\prime\prime}$ exists as a signed measure with finite mass, 
\item $\xi_{\ini}=0$.
\end{enumerate}
\end{assumption}
The third condition is a mere convention while the first one implies the continuity of $p_\ini$ since $u_\ini$ is piecewise continuous and exhibits a jump of height $2$ at the initial interface located at $x=\xi_\ini$. The second condition can --- with a slight abuse of notation -- be written as
\begin{align}
\label{eq:Qini}
\bnorm{p_{\ini}^{\prime\prime}}_1<\infty
\end{align} 
and allows $\partial_x p_\ini$ to have jumps, see the initial data of the numerical simulations in Figure \ref{fig:Simulation1} for a typical example.  Thanks to \eqref{eq:Qini},  we are able to bound the discretization error of the regular part explicitly in Lemma \ref{konv: q} but a similar convergence result holds under weaker assumptions.  We also emphasize that all other estimates derived below depend on the initial data only via the parameter $\al$ from Definition \ref{defi: zul}, which bounds the interface speed according to Corollary \ref{cor: interfacespeed}.
\par
We already mentioned that \eqref{sc: q} is equivalent to an implicit time scheme for the homogeneous diffusion equation and this implies the following convergence result.
\begin{lemma}[\bf convergence of $q_\eps$]
\label{konv: q}
There exists a universal constant $C$ such that
\begin{align*}
\bnorm{q_\eps\pair{t}{\cdot}-q_0\pair{t}{\cdot}}_\infty\leq C\,\min\left\{\eps,\,\frac{\eps^2}{\sqrt{t}}\right\}\,\bnorm{p_\ini^{\prime\prime}}_1\,,
\end{align*}
where $q_0$ denotes the unique smooth solution to the initial value problem
\begin{align}
\notag
\partial_t q_0=\partial_x^2 q_0\,,\qquad q_0\pair{0}{x}=p_\ini\at{x}
\end{align}
with $x\in\Rset$ and $t\in\ccinterval{0}{T}$.
\end{lemma}
\begin{proof}
The properties of the heat kernel and the assumptions on $p_\ini$ imply the smoothness of $q_0$. Moreover, in the appendix we show that
\begin{align}
\label{eq:EstQ}
\Bnorm{q_\eps\at{t^n}-q_0\at{t^n}}_\infty=
\Bnorm{{\underbrace{\at{g_{\eps}\ast\ldots\ast g_{\eps}}}_{\substack{\text{$n$ times}}}}\ast q_\ini-{\underbrace{\at{h_\eps\ast\ldots\ast h_\eps}}_{\substack{\text{$n$ times}}}}\ast q_\ini}_\infty\leq \frac{C\,\eps}{\sqrt{n}}\,\bnorm{q_\ini^{\prime\prime}}_1
\end{align}
holds for all $n\geq 1$, where the function $h_\eps$ is defined below in \eqref{eq:HeatKernel1}. The claim is now implied by \eqref{eq: tn couple} and a simple interpolation argument.
\end{proof}
The convergence of the interface curve follows from the upper bound for the interface speed in Lemma \ref{cor: interfacespeed}.
\begin{lemma}[\bf compactness of interface curve]
\label{lem: interface}
The family $(\xi_{\eps})_{\eps}$ is compact in $\fspaceL^{\infty}\left([0,T]\right)$ and 
\begin{align*}
\sup_{0\leq t\leq T}\abs{\xi_\eps\at{t}-\xi_0\at{t}}\xrightarrow{\;\;\eps\to 0\;\;}0
\end{align*}
holds along subsequences. Moreover, any limit function $\xi_0$ is Lipschitz continuous as it satisfies
\begin{align*}
0\leq \dot{\xi}_0\at{t}\leq\frac{\alpha}{2}
\end{align*}
for almost all $t\in\ccinterval{0}{T}$.
\end{lemma}
\begin{proof} 
The piecewise linear interpolation 
\begin{align*}
\bar{\xi}_{\eps}\at{t}:=\sum_{m=0}^{N}\left[\xi^m+\at{t-m\,\eps^2}\,\frac{\xi^{m+1}-\xi^m}{\eps^2}\right]\chi_{\cointerval{m\,\eps^2}{\at{m+1}\,\eps^2}}\at{t}\,,
\end{align*}
satisfies $\ol\xi_\eps\at{0}=\xi_\eps\at{0}=0$ as well as
\begin{align*}
0\leq \dot{\ol\xi}_\eps\at{t}\leq\frac{ \xi^n-\xi^{n+1}}{\eps^2}\,,\qquad \quad\babs{\ol{\xi}_\eps\at{t}-\xi_\eps\at{t}}\leq \xi^n-\xi^{n+1} \qquad\text{for}\quad t\in\cointerval{t^n}{t^{n+1}}
\end{align*}
thanks to $\xi^{n+1}\leq\xi^n$ and \eqref{eq:IdentificationY}. Since the right hand sides are uniformly bounded by Lemma \ref{cor: interfacespeed}, all assertions follow from the compact embedding $\fspaceW^{1,\infty}\bat{\ccinterval{0}{T}}\hookrightarrow \fspaceC\bat{\ccinterval{0}{T}}$, see for instance \cite[Proposition 8.4 and Theorem 8.8]{Bre11}, and the estimate $\xi_\eps\at{t_2}\leq \xi_\eps\at{t_1}$, which holds for any $\eps$ and all times $t_1, t_2$ with $0\leq t_1<t_2\leq T$.  
\end{proof}
%
%
\subsection{Smallness of  the negligible fluctuations}
\label{sect:negligible}
%
%
As sketched in Figure \ref{fig: strategy}, we successively split off the terms $f_{\eps,\,\negl_i}$ from the global fluctuations $f_\eps$ and show that these are negligible in the sense of
\begin{align*}
\norm{f_{\eps,\,\negl _i}}_\infty\leq C\,\al\,\eps\,.
\end{align*}
%
%
\begin{figure}[ht!]
\centering{
\includegraphics[width=0.5\textwidth]{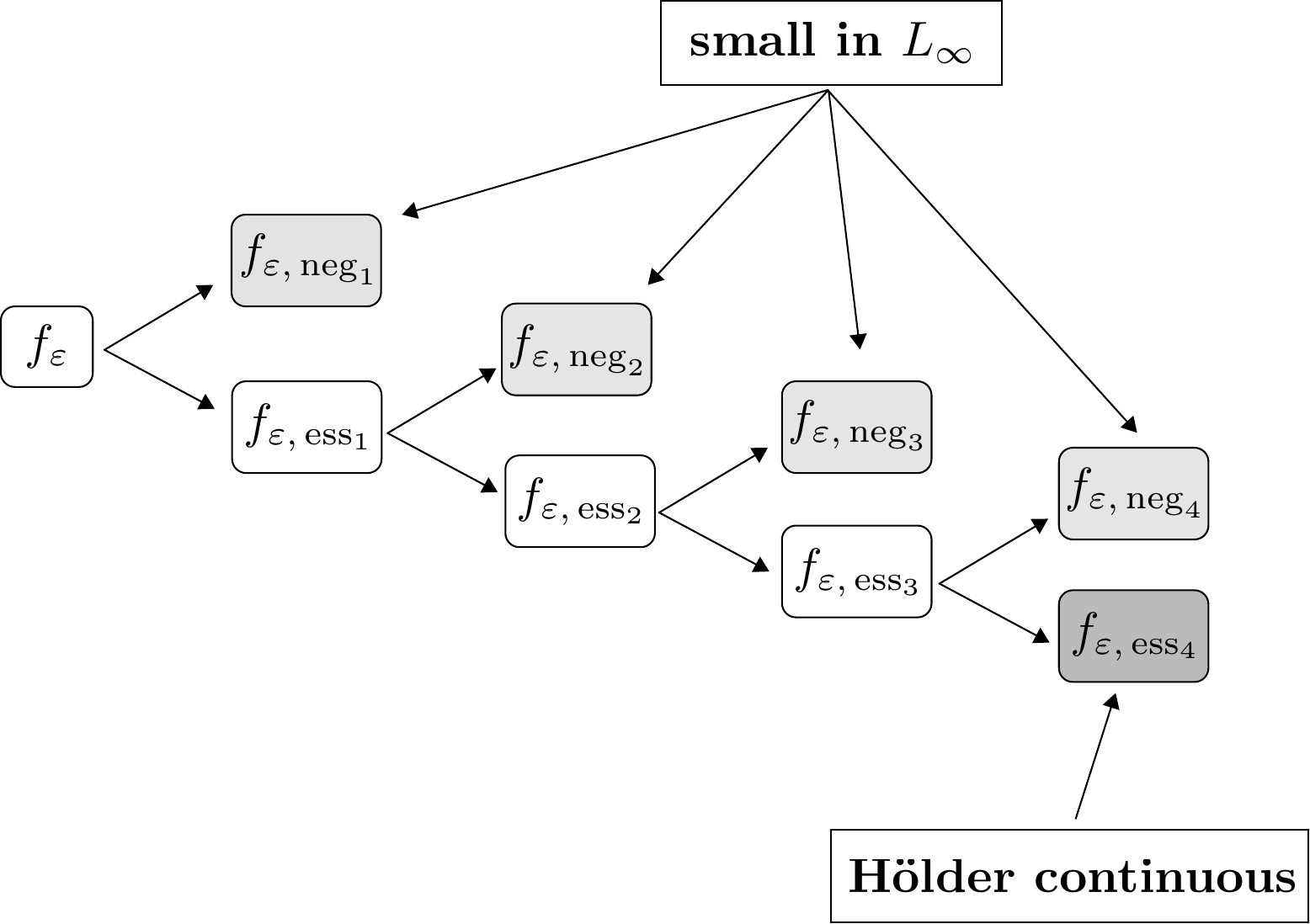}
}%
\caption{The strategy of splitting the global fluctuations $f_\eps$: Four times we split off negligible contributions $f_{\eps,\,\negl_i}$ and finally obtain the essential term $f_{\eps,\,\ess_4}$}
\label{fig: strategy}
\end{figure}
%
%
To control the global fluctuations, we approximate the local fluctuations in the representation formula \eqref{eq:DefF} in a first step by the convolution kernel $g_\eps$ and afterwards by the function 
\begin{align}
\label{eq:HeatKernel1}
h_\eps\at{x}:=G_0\pair{\eps^2}{x}\,,
\end{align}
which represents the heat kernel 
\begin{align}
\label{eq:HeatKernel0}
G_0\pair{t}{x}=\frac{1}{\sqrt{4\,\pi\,t\,}}\exp\at{-\frac{x^2}{4\,t}}
\end{align}
at time $t=\eps^2$. More precisely, we define the first and the second essential fluctuations by
\begin{align*}
f_{\eps,\,\ess_1}\pair{t}{x}&:=\sum_{i=1}^{n} \underbrace{\at{g_{\eps}\ast\ldots\ast g_{\eps}}}_{\substack{\text{$n{-}i$ times}}}\ast r_{\ess}^i\at{x}\\&=\sum_{i=1}^{n} 2\,\eps^2\,\frac{\xi^{i-1}-\xi^{i}}{\eps^2}\underbrace{\at{g_{\eps}\ast\ldots\ast g_{\eps}}}_{\substack{\text{$n{-}i{+}1$ times}}}\at{x-\frac{1}{2}\at{\xi^i+\xi^{i-1}}}
\,,
\end{align*}
and
\begin{align*}
f_{\eps,\,\ess_2}\pair{t}{x}&:=
\sum_{i=1}^{n} 2\,\eps^2\,\frac{\xi^{i-1}-\xi^{i}}{\eps^2}\underbrace{\at{h_\eps\ast\ldots\ast h_\eps}}_{\substack{\text{$n{-}i{+}1$ times}}}\at{x-\frac{1}{2}\at{\xi^i+\xi^{i-1}}}
\\&=\sum_{i=1}^{n} 2\,\eps^2\,\frac{\xi^{i-1}-\xi^{i}}{\eps^2}\,G_{0}\pair{\at{n-i+1}\,\eps^2}{x-\frac{1}{2}\,\at{\xi^i+\xi^{i-1}}}\,.
\end{align*} 
Here, the two time variables $n$ and $t$ are coupled as in \eqref{eq: tn couple} and we already inserted the definition \eqref{eq:DefREss} as well as the identify
\begin{align*}
\underbrace{h_\eps\ast\ldots\ast h_\eps}_{\substack{\text{$n$ times}}}=G_0\pair{n\,\eps^2}{\cdot}\,,
\end{align*}
which follows from elementary properties of $G_0$. The next result ensures that the corresponding error terms
\begin{align}
\label{eq:DefFN12}
f_{\eps,\,\negl_1}\pair{t}{x}:=f_{\eps}\pair{t}{x}-f_{\eps,\,\ess_1}\pair{t}{x}\,,\qquad
f_{\eps,\,\negl_2}\pair{t}{x}:=f_{\eps,\,\ess_1}\pair{t}{x}-f_{\eps,\,\ess_2}\pair{t}{x}
\end{align}
are in fact negligible as $\eps\to0$ and relies on the key estimate
\begin{align}
\label{eq: key ingedrient}
\Bnorm{\underbrace{\at{g_{\eps}\ast\ldots\ast g_{\eps}}}_{\substack{\text{$n$ times}}}-\underbrace{\at{h_\eps\ast\ldots\ast h_\eps}}_{\substack{\text{$n$ times}}}}_\infty\leq \frac{C}{\eps\,n^{3/2}}\,,
\end{align}
which we prove in the appendix by means of Fourier analysis.
\begin{lemma}[\bf smallness of $\boldsymbol{f_{\eps,\,\negl_1}}$ and $\boldsymbol{f_{\eps,\,\negl_2}}$]
\label{lem: n12}
The estimates
\begin{align*}
\bnorm{f_{\eps,\,\negl_1}}_\infty\leq C\,\al\,\sqrt{T}\,\eps\,,\qquad 
\bnorm{f_{\eps,\,\negl_2}}_{\infty}\leq C\,\al\,\eps
\end{align*}
are satisfied by the functions from \eqref{eq:DefFN12}. 
\end{lemma}
\begin{proof}
Corollary \ref{cor: interfacespeed}, Lemma $\ref{lem: eff}$, and \eqref{eq: tn couple} provide
\begin{align*}
\babs{f_{\eps,\,\negl_1}\pair{t}{x}}\leq\sum_{i=1}^{n} \underbrace{\at{g_{\eps}\ast\ldots\ast g_{\eps}}}_{\substack{\text{$n{-}i$ times}}}\ast \babs{r^{i}\at{x}-r^i_{\ess}\at{x}}\leq C\,\al\,\eps\,\sum_{i=1}^{n} \underbrace{\at{g_{\eps}\ast\ldots\ast g_{\eps}}}_{\substack{\text{$n{-}i$ times}}}\ast \babs{r^i_{\ess}\at{x}}
\end{align*}
and this gives
\begin{align*}
\bnorm{f_{\eps,\,\negl_1}\pair{t}{\cdot}}_\infty\leq C\,\al\,\eps\,\bnorm{f_{\eps,\,\ess_1}\pair{t}{\cdot}}_\infty
\end{align*}
since the terms $r^i_\ess$ are positive. Moreover, we have
\begin{align*}
\bnorm{f_{\eps,\,\ess_1}\pair{t}{\cdot}}_\infty\leq \bnorm{f_{\eps,\,\negl_2}\pair{t}{\cdot}}_\infty+\bnorm{f_{\eps,\,\ess_2}\pair{t}{\cdot}}_\infty
\end{align*}
with
\begin{align*}
\bnorm{f_{\eps,\,\negl_2}\pair{t}{\cdot}}_\infty&\leq C\,\al\,\eps^2\sum_{i=1}^{n} \Bnorm{\underbrace{\at{g_{\eps}\ast\ldots\ast g_{\eps}}}_{\substack{\text{$n{-}i{+}1$ times}}}-\underbrace{\at{h_\eps\ast\ldots\ast h_\eps}}_{\substack{\text{$n{-}i{+}1$ times}}}}_\infty\\&
\leq C\,\al\,\eps^2\sum_{i=1}^{n}\frac{1}{\eps\,\at{n-i-1}^{3/2}}=
 C\,\al\,\eps\sum_{i=1}^{n}\frac{1}{i^{3/2}}
\end{align*}
and
\begin{align*}
\bnorm{f_{\eps,\,\ess_2}\pair{t}{\cdot}}_\infty&\leq C\,\al\,\eps^2\sum_{i=1}^{n}\Bnorm{\underbrace{\at{h_\eps\ast\ldots\ast h_\eps}}_{\substack{\text{$n{-}i{+}1$ times}}}}_\infty
\leq C\,\al\,\eps^2\sum_{i=1}^{n}\Bnorm{G_0\bpair{\at{n-i+1}\,\eps^2}{\cdot}}_\infty
\\&\leq C\,\al\,\eps^2\,\sum_{i=1}^{n}\frac{1}{\sqrt{\at{n-i+1}\,\eps^2}}=C\,\al\,\eps\,\sum_{i=1}^{n}\frac{1}{i^{1/2}}
\end{align*}
thanks  to \eqref{eq: key ingedrient} and the standard estimate $\bnorm{G_0\pair{t}{\cdot}}_\infty\leq C\, t^{-1/ 2}$. The sums on the right hand sides of the last two estimates can be bounded by Riemann integrals via
\begin{align*}
\sum_{i=1}^{n}\frac{1}{i^{3/2}}\leq 1+\int\limits_{1}^{\infty}\frac{\dint{y}}{y^{3/2}}=3\,,\qquad\sum_{i=1}^{n}\frac{1}{i^{1/2}}\leq \int\limits_{0}^{n}\frac{\dint{y}}{y^{1/2}}\leq 2\,n^{1/2}\leq C\,\sqrt{T}/\eps
\end{align*}
and the claim follows by combining all partial results.
\end{proof}
%
%
\begin{figure}[ht!]%
\centering{%
\includegraphics[width=0.4\textwidth]{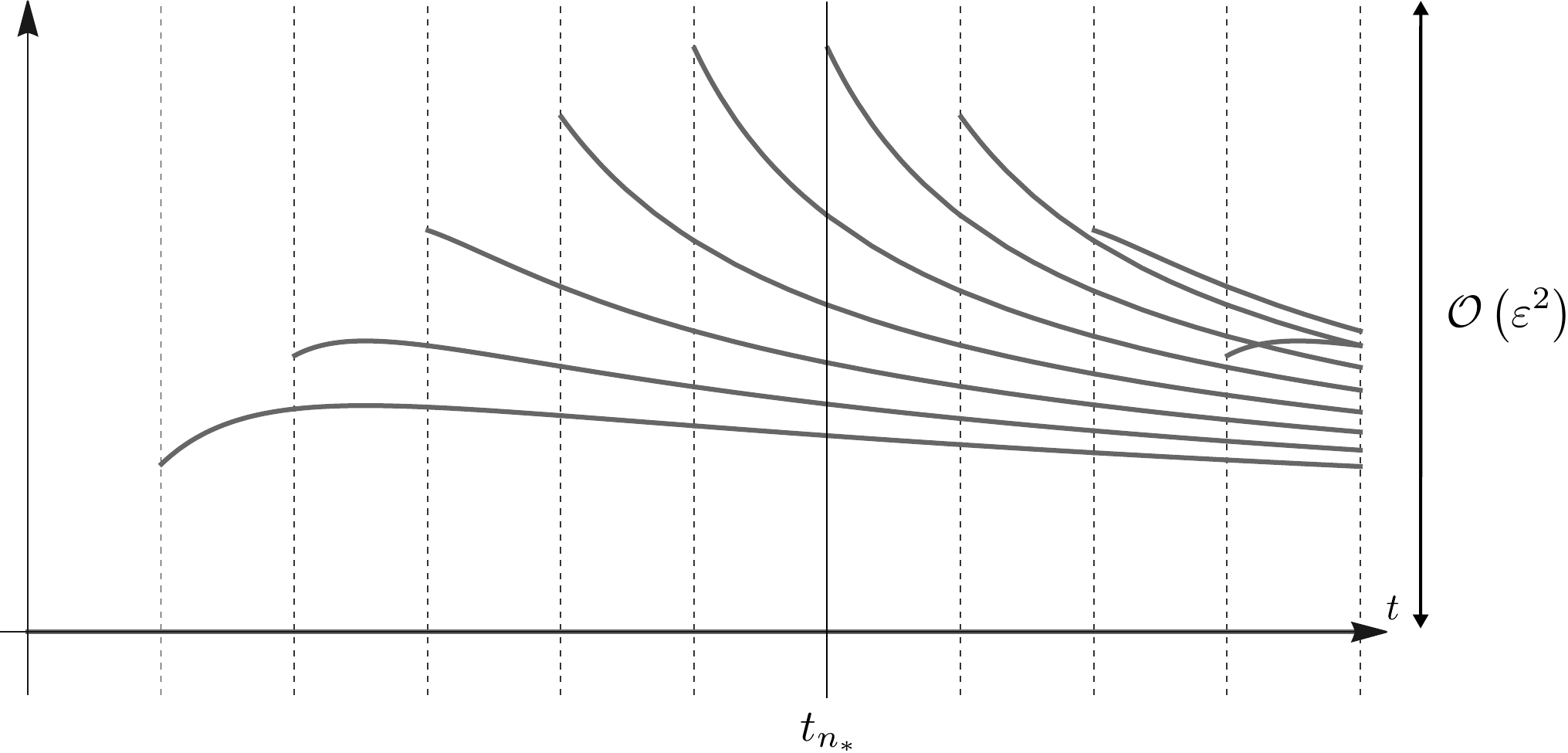}%
\hspace{0.05\textwidth}%
\includegraphics[width=0.4\textwidth]{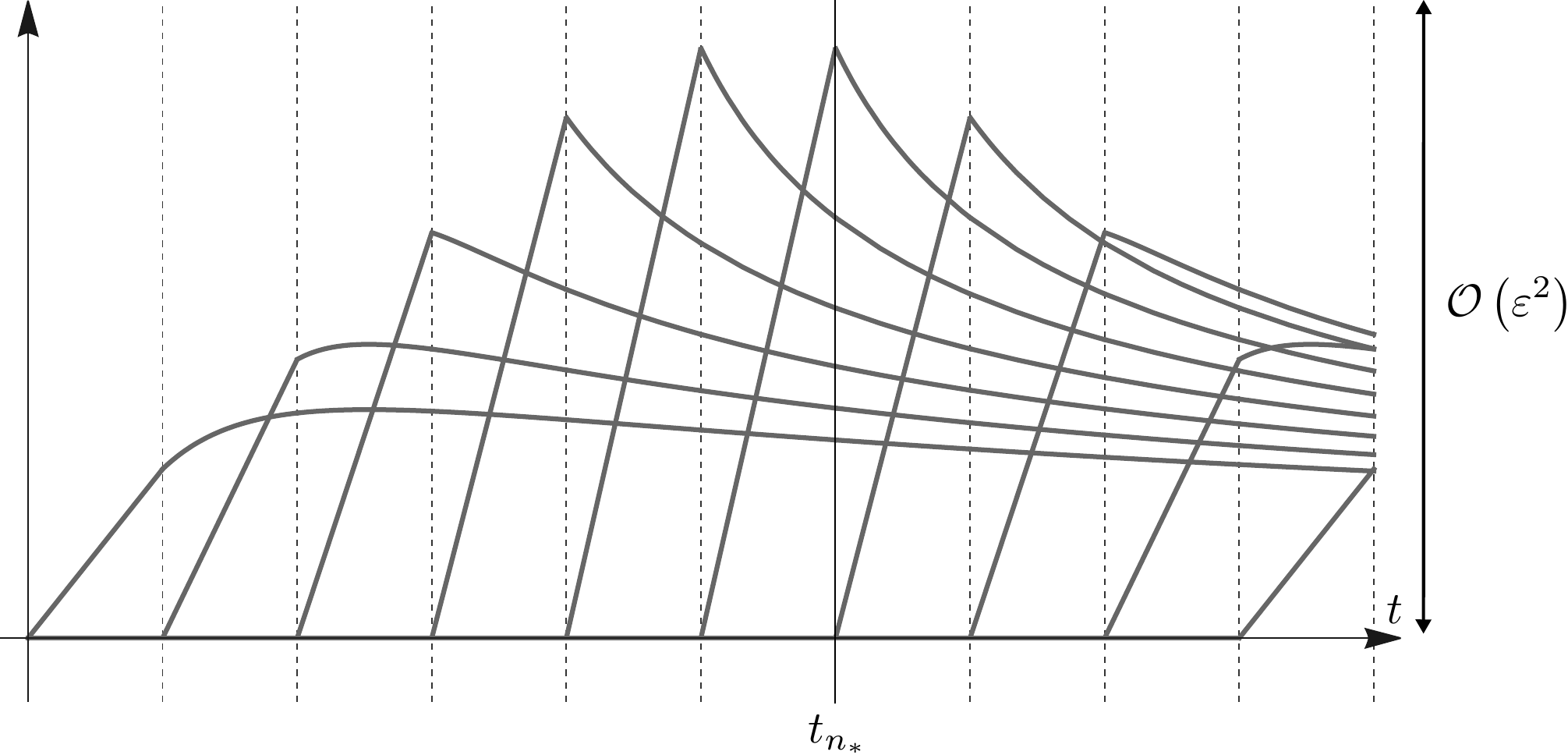}%
}%
\caption{Schematic illustration of the additive contributions to the essential fluctuations $f_{\eps,\,\ess_3}$ (left panel) and $f_{\eps,\,\ess_4}$ (right panel) as functions of $t$ with fixed $x=\xi^{n_*}$. The vertical lines represents the times $t^n=n\,\eps^2$ at which the number of summands increases according to \eqref{eq: tn couple} and \eqref{eq:DefFE3}$+$\eqref{eq:DefFE4}. In particular, $f_{\eps,\,\ess_4}$ is continuous with respect to $t$ while $f_{\eps,\,\ess_3}$ exhibits many small jumps.}
\label{fig:FE34loc}
\end{figure}
%
%
The essential fluctuations $f_{\eps,\,\ess_2}$ are piecewise constant in time and each summand exhibits a temporal jump at any multiple of $\eps^2$. We regularize the fluctuations in the third step by using $G_0\pair{t}{\cdot}$ instead of $G_0\pair{n\,\eps^2}{\cdot}$ but this does not eliminate all jumps since at any multiple of $\eps^2$ an additional summand has to be taken into account due to the increase in $n$. In the fourth step we therefore replace $G_0$ by the regularized heat kernel 
\begin{align}
\label{eq: Heps}
H_{\eps}\pair{t}{x}=
\begin{cases}
0 & \text{for } t \leq 0\,,\\
\D\frac{t}{\eps^2}\,G_{0}(\eps^2,x) & \text{for } 0 \leq t \leq \eps^2\,,\\
G_{0}\pair{t}{x}& \text{for } t \geq \eps^2\,,
\end{cases}
\end{align}
which is continuous with respect to all $x\in\Rset$ and all times $t\in\Rset$. Specifically, we now consider the fluctuations 
\begin{align}
\label{eq:DefFE3}
f_{\eps,\,\ess_3}\pair{t}{x}:=\sum_{i=1}^{n} 2\,\eps^2\,\frac{\xi^{i-1}-\xi^{i}}{\eps^2}\,G_{0}\pair{t-i\,\eps^2+\eps^2}{x-\frac{1}{2}\,\at{\xi^i+\xi^{i-1}}}
\end{align}
and
\begin{align}
\label{eq:DefFE4}
f_{\eps,\,\ess_4}\pair{t}{x}:=\sum_{i=1}^{\infty} 2\,\eps^2\,\frac{\xi^{i-1}-\xi^{i}}{\eps^2}\,H_{\eps}\pair{t-i\,\eps^2+\eps^2}{x-\frac{1}{2}\,\at{\xi^i+\xi^{i-1}}}
\end{align}
along with the error terms
\begin{align}
\label{eq:DefFN34}
f_{\eps,\,\negl_3}\pair{t}{x}:=f_{\eps,\,\ess_2}\pair{t}{x}-f_{\eps,\,\ess_3}\pair{t}{x}\,,\qquad
f_{\eps,\,\negl_4}\pair{t}{x}:=f_{\eps,\,\ess_3}\pair{t}{x}-f_{\eps,\,\ess_4}\pair{t}{x}\,,
\end{align}
where $n$ and $t$ are still coupled by \eqref{eq: tn couple}. Notice also that $f_{\eps,\,\ess_4}$ is continuous with respect to $t$, see in Figure \ref{fig:FE34loc} for illustration, and that the series in its definition is actually a finite sum for any fixed $t$. The H\"older estimate
\begin{align}
\label{eq:Hoelder}
\babs{H_{\eps}\pair{t_2}{x_2}-H_{\eps}\pair{t_1}{x_1}}\leq C\,\at{\frac{\abs{t_2-t_1}^{\gamma}}{\max\{\eps^2,t_1\}^{\frac{1}{2}+\gamma}}+\frac{\sqrt{\abs{x_2-x_1}}}{\max\{\eps^2,\,t_1\}^{3/4}}}
\end{align}
plays an important role in subsequent analysis and holds with arbitrary exponent $0\leq\ga\leq 1$ for all $0\leq t_1\leq t_2<\infty$ and $x_1,x_2\in\Rset$. The proof of \eqref{eq:Hoelder} exploits the piecewise linear behavior of $H_\eps\pair{t_1}{\cdot}$ for $t_1\leq \eps^2$ and that $G_0$ from \eqref{eq:HeatKernel0} satisfies analogous estimates for $t_1\geq\eps^2$.
\begin{lemma}[\bf smallness of $\boldsymbol{f_{\eps,\,\negl_3}}$ and $\boldsymbol{f_{\eps,\,\negl_4}}$]
\label{lem: n34}
We have
\begin{align*}
\norm{f_{\eps,\,\negl_3}}_{\infty}\leq C\,\al\,\eps\,,\qquad \norm{f_{\eps,\,\negl_4}}_{\infty}\leq C\,\al\,\eps
\end{align*}
for the functions from \eqref{eq:DefFN34}.
\end{lemma}
\begin{proof}
Using the temporal H\"older estimate
\begin{align*}
\bnorm{G_0\pair{t_2}{\cdot}-G_0\pair{t_1}{\cdot}}_\infty\leq C\, \frac{\abs{t_2-t_1}}{\,\min\{t_1,\,t_2\}^{3/2}\,}
\end{align*} 
and Corollary \ref{cor: interfacespeed} we obtain
\begin{align*}
\bnorm{f_{\eps,\,\negl_3}\pair{t}{\cdot}}_\infty&\leq C\,\al\,\eps^2\,\sum_{i=1}^{n} \norm{G_{0}\pair{\at{n-i+1}\,\eps^2}{\cdot}-G_{0}\pair{t-i\,\eps^2+\eps^2}{\cdot}}_\infty\\
&\leq C\,\al\,\eps^2 \sum_{i=1}^{n} \frac{t-n\,\eps^2}{\bat{\at{n-i+1}\,\eps^2}^{3/2}}\leq C\,\al\,\eps\,\sum_{i=1}^{n} \frac{1}{i^{3/2}}\leq C\,\al\,\eps
\end{align*}
and the first desired estimate follows by taking the supremum with respect to $t\in\ccinterval{0}{T}$. Moreover, \eqref{eq: Heps} implies 
\begin{align*}
f_{\eps,\,\negl _4}\pair{t}{x}= 2\,\eps^2\,\frac{\xi^{n}-\xi^{n+1}}{\eps^2}\,H_{\eps}\Bpair{t-n\,\eps^2+\eps^2}{x-\frac{1}{2}\,\at{\xi^{n+1}+\xi^{n}}}
\end{align*}
since the contributions for all indices $i\neq n+1$ vanish by construction. This implies
\begin{align*}
\bnorm{f_{\eps,\,\negl_4}}_
\infty&\leq C\,\al\,\eps^2\,\norm{H_\eps}_\infty=C\,\al\,\eps^2\,G_0\pair{\eps^2}{0}\leq C\,\al\,\eps /\sqrt{\eps^2}=C\,\al\,\eps
\end{align*}
and completes the proof.
\end{proof}
%
%
\subsection{H\"older continuity of the essential fluctuations}
\label{sect:essential}
%
%
It remains to establish the uniform H\"older continuity of the essential fluctuations $f_{\eps,\,\ess_4}$.
\begin{lemma}[\bf H\"older continuity in time of $\boldsymbol{f_{\eps,\,\ess_4}}$]
\label{lem: holder ess time}
For any $0<\gamma<1/2$, there exist a constant $C$ independent of $\eps$ and $T$ such that
\begin{align}
\notag
\babs{f_{\eps,\,\ess_4}\pair{t_2}{x}-f_{\eps,\,\ess_4}\pair{t_1}{x}}\leq C\,\al\,\babs{t_2-t_1}^\gamma\,T^{1/2-\gamma}
\end{align}
holds for any $x\in\Rset$ and all $0\leq t_1\leq t_2\leq T$.
\end{lemma}
\begin{proof}
For given $0\leq t_1\leq t_2\leq T$ we choose $n_1, n_2$ such that $t_j\in \cointerval{n_j\,\eps^2}{\at{n_j+1}\,\eps^2}$ and fix $x\in\Rset$ arbitrarily. Using \eqref{eq:DefFE4} and Corollary \ref{cor: interfacespeed} we deduce
\begin{align*}
&\babs{f_{\eps,\,\ess_4}\pair{t_2}{x}-f_{\eps,\,\ess_4}\pair{t_1}{x}}\leq D_1+D_2
\end{align*}
with
\begin{align*}
D_1:= C\,\al\,\eps^2\sum_{i=1}^{n_1+1}\abs{H_\eps\Bpair{t_2-i\,\eps^2+\eps^2}{x-\frac{1}{2}\,\at{\xi^i+\xi^{i-1}}}-H_\eps\Bpair{t_1-i\,\eps^2+\eps^2}{x-\frac{1}{2}\,\at{\xi^i+\xi^{i-1}}}}
\end{align*}
and 
\begin{align*}
D_2:= C\,\al\,\eps^2\sum_{i=n_1+2}^{n_2+1}\abs{H_\eps\Bpair{t_2-i\,\eps^2+\eps^2}{x-\frac{1}{2}\,\at{\xi^i+\xi^{i-1}}}}\,,
\end{align*}
where the sum $D_2$ vanishes in the case of $n_1=n_2$. 
\par\underline{\emph{Estimate for $D_1$}}\,:
 The H\"older estimate \eqref{eq:Hoelder} implies
\begin{align*}
D_1&\leq C\,\al\,\eps^2\,\babs{t_2-t_1}^\gamma\,\sum_{i=1}^{n_1+1}\frac{1}{\max\bpair{\eps^2}{t_1-i\,\eps^2+\eps^2}^{\gamma+1/2}}\\
&\leq C\,\al\,\eps^2\,\babs{t_2-t_1}^\gamma\,\at{\frac{1}{\bat{\eps^2}^{\gamma+\frac{1}{2}}}+\sum_{i=1}^{n_1}\frac{1}{\bat{\at{n_1-i+1}\,\eps^2}^{\gamma+1/2}}}\\
&\leq C\,\al\,\eps^2\abs{t_2-t_1}^\gamma\,\at{\frac{1}{\bat{\eps^2}^{\gamma+\frac{1}{2}}}+\sum_{i=1}^{n_1}\frac{1}{\bat{i\,\eps^2}^{\gamma+1/2}}}\\
&\leq  C\,\al\,\abs{t_2-t_1}^\gamma\,\eps^{1-2\,\gamma}\,\sum_{i=1}^{n_1}\frac{1}{i^{\gamma+1/2}}
\\&\leq
C\,\al\,\abs{t_2-t_1}^\gamma\,T^{1/2-\gamma}\,,
\end{align*}
where the last estimate follows from the Riemann sum approximation
\begin{align*}
\sum_{i=1}^{n_1}\frac{1}{i^{\gamma+1/2}}\leq \int\limits_{0}^{n_1-1}\frac{\dint{y}}{y^{\gamma+1/2}}\leq C\,n_1^{1/2-\gamma}
\end{align*}
due to $n_1\leq T/\eps^2$. 
\par\underline{\emph{Estimate for $D_2$}}\,: 
Assuming $n_2\geq n_1+2$ and using \eqref{eq: Heps} as well as $\norm{G_0\pair{t}{\cdot}}_\infty\leq C/\sqrt{t}$ we get
\begin{align*}
D_2&\leq C\,\al\,\eps^2\,\frac{t_2-n_2\,\eps^2}{\eps^2}\,\bnorm{G_0\pair{\eps^2}{\cdot}}_\infty+
C\,\al\,\eps^2\sum_{i=n_1+2}^{n_2} \bnorm{G_0\pair{t_2-i\,\eps^2+\eps^2}{\cdot}}_\infty\\
&\leq C\,\al\,\at{\eps+\eps^2\sum_{i=n_1+2}^{n_2}\frac{1}{\sqrt{\at{n_2-i+1}\,\eps^2}}}
\leq C\,\al\,\eps\,\sum_{i=1}^{n_2-n_1-1}\frac{1}{\sqrt{i}}\leq C\,\al\,\eps\,\sqrt{n_2-n_1}\\
&\leq C\,\al\,\sqrt{\abs{t_2-t_1}}=C\,\al\,\abs{t_2-t_1}^{\gamma}\,\abs{t_2-t_1}^{+1/2-\gamma}%
\leq C\,\al\,\abs{t_2-t_1}^{\gamma}\,T^{1/2-\gamma}\,.
\end{align*}
For $n_2=n_1+1$ we have
\begin{align*}
0\leq t_2- n_2\,\eps^2\leq t_2-t_1\leq 2\,\eps^2 
\end{align*}
and obtain the estimate
\begin{align*}
D_2&\leq C\,\al\,\eps^2\,\frac{t_2-n_2\,\eps^2}{\eps^2}\,\bnorm{G_0\pair{\eps^2}{\cdot}}_\infty\leq C\,\al\,\frac{t_2-t_1}{\eps}
\leq C \,\al\,\babs{t_2-t_1}^\ga\,\frac{\babs{t_2-t_1}^{1-\ga}}{\eps}
\\&\leq C \,\al\,\babs{t_2-t_1}^\ga\, T^{1/2-\gamma}\,,
\end{align*}
where the last inequality holds since $T\ge\eps^2$ implies $\at{\eps^2}^{1-\ga}\,\eps^{-1}=\eps^{1-2\,\ga}\leq T^{1/2-\gamma}$.
\end{proof}
\begin{lemma}[\bf H\"older continuity in space of $\boldsymbol{f_{\eps,\,\ess_4}}$]
\label{lem: holder ess space}
We have
\begin{align}
\notag
\babs{f_{\eps,\,\ess_4}\pair{t}{x_2}-f_{\eps,\,\ess_4}\pair{t}{x_1}}\leq C\,\al\,T^{1/4}\,\sqrt{\abs{x_2-x_1}}
\end{align}
for any $0\leq t\leq T$ and all $x_1$, $x_2\in\Rset$.
\end{lemma}
\begin{proof}
We fix $x_1, x_2\in \Rset$ and derive
\begin{align*}
\babs{f_{\eps,\,\ess_4}\pair{t}{x_2}-f_{\eps,\,\ess_4}\pair{t}{x_1}}\leq D
\end{align*}
with
\begin{align*}
D:=C\,\al\,\eps^2 \sum_{i=1}^{n+1} \abs{H_\eps\Bpair{t-i\,\eps^2+\eps^2}{x_2-\frac{1}{2}\at{\xi^i+\xi^{i-1}}}-H_\eps\Bpair{t-i\,\eps^2+\eps^2}{x_1-\frac{1}{2}\at{\xi^i+\xi^{i-1}}}}
\end{align*}
from  \eqref{eq:DefFE4} and Corollary \ref{cor: interfacespeed}, where $\pair{t}{n}$ satisfy \eqref{eq: tn couple}. The H\"older estimate \eqref{eq:Hoelder} provides
\begin{align*}
D&\leq C\,\al\,\sqrt{\abs{x_1-x_2}}\,\eps^2\at{\frac{1}{\bat{\eps^2}^{3/4}}+\sum_{i=1}^{n} \frac{1}{\bat{t-i\,\eps^2+\eps^2}^{3/4}}}\\
&\leq C\,\al\,\sqrt{\abs{x_1-x_2}}\,\eps^{1/2}\at{1+\sum_{i=1}^{n} \frac{1}{\bat{n-i+1}^{3/4}}}\\
&\leq C\,\al\,\sqrt{\abs{x_1-x_2}}\,\eps^{1/2}\int\limits_{0}^{n}\frac{1}{y^{3/4}}\dint{y}\,,
\end{align*}
where the last estimate is a Riemann sum approximation and yields the desired result thanks to $\eps^{1/2}\, n^{1/4}\leq T^{1/4}$.
\end{proof}
For later use we establish uniform estimates for the space dependence of the essential fluctuations.
\begin{lemma}[\bf boundedness and regularity of $\boldsymbol{f_{\eps,\,\ess_4}}$]
\label{lem: reg}
We have%
\begin{align}
\notag
\bnorm{f_{\eps,\,\ess_4}\pair{t}{\cdot}}_\infty\leq C\,\al\,\sqrt{T}\,,\qquad 
\bnorm{\partial_x f_{\eps,\,\ess_4}\pair{t}{\cdot}}_{2}\leq C\,\al\,T^{1/4}
\end{align}
for all $t\in\ccinterval{0}{T}$.
\end{lemma}
\begin{proof}
\underline{\emph{First part}}\,: 
Using the definition of $H_\eps$ in \eqref{eq: Heps} as well as the upper bound for the interface speed from Corollary \ref{cor: interfacespeed} we obtain
\begin{align*}
\babs{f_{\eps,\,\ess_4}\pair{t}{x}}&\leq C\,\al\,\eps^2 \sum_{i=1}^{n+1} \abs{H_{\eps}\Bpair{t-i\eps^2+\eps^2}{x-\frac{1}{2}\,\at{\xi^i+\xi^{i-1}}}}\\
&\leq C\,\al\,\eps^2\,\frac{t-n\,\eps^2}{\eps^2}\,\bnorm{G_0\pair{\eps^2}{\cdot}}+C\,\al\,\eps^2 \sum_{i=1}^{n} \bnorm{G_0\pair{t-i\,\eps^2+\eps^2}{\cdot}}_\infty\,,
\end{align*}
where $n$ and $t$ are coupled by \eqref{eq: tn couple}. With $\norm{G_0\pair{t}{\cdot}}_\infty\leq C/\sqrt{t}$ we therefore get
\begin{align*}
\babs{f_{\eps,\,\ess_4}\pair{t}{x}}&\leq C\,\al\,\eps+C\,\al\,\eps^2\,\sum_{i=1}^{n}\frac{1}{\at{t-i\,\eps^2+\eps^2}^{1/2}}
\leq 
C\,\al\,\eps+C\,\al\,\eps^2\,\sum_{i=1}^{n}\frac{1}{\at{n\,\eps^2-i\,\eps^2+\eps^2}^{1/2}}
\\&\leq C\,\al\,\eps+C\,\al\,\eps\,\sum_{i=1}^{n}\frac{1}{i^{1/2}}
\end{align*}
and a Riemann sum approximation provides the first estimate due to $\eps\leq \sqrt{T}$.
\par\underline{\emph{Second part}}\,: 
With $\norm{\partial_x G_0\pair{t}{\cdot}}_2\leq C\,t^{-3/2}$ and analogously to the first part we find
\begin{align*}
\bnorm{\partial_x f_{\eps,\,\ess_4}\pair{t}{\cdot}}_2 &\leq C\,\al\,\eps^{1/2}+C\,\al\,\eps^{1/2}\,\sum_{i=1}^{n}\frac{1}{\at{n-i+1}^{3/4}}=C\,\al\,\eps^{1/2}+C\,\al\,\eps^{1/2}\,\sum_{i=1}^{n}\frac{1}{i^{3/4}}
\end{align*}
and this yields the second desired estimate.
\end{proof}
%
%
\subsection{Passage to the limit}
%
%
The uniform estimates from the previous sections provide via \eqref{eq:IdentificationX} the compactness of the time-discrete data produced by Scheme \ref{eq: IE}.
\begin{corollary}[\bf compactness]
\label{cor: conv}
There exist subsequences for $\eps\to 0$ with the following properties.
\begin{enumerate}
\item 
We have $\xi_\eps\to \xi_0\text{ in } \fspaceL^{\infty}\bat{\ccinterval{0}{T}}$ for a limit function $\xi_0\in\fspaceW^{1,\,\infty}\bat{\ccinterval{0}{T}}$.
\item 
The function $q_\eps$ converges to $q_0\in \fspaceC\bat{\ccinterval{0}{T}\times\Rset}$ in $\fspaceL^{\infty}\bat{\ccinterval{0}{T}\times\Rset}$ and therefore pointwise almost everywhere.
\item 
There exist a limit function $f_0\in\fspaceC\bat{\ccinterval{0}{T}\times\Rset}$, such that $f_\eps$ converges pointwise almost everywhere to $f_0$ as well as in $\fspaceL^\infty\at{K}$ on every compact subset $K\subset\ccinterval{0}{T}\times\Rset$.
\end{enumerate}
In particular, we have
\begin{align*}
p_{\eps}=q_\eps-f_\eps\quad \xrightarrow{\;\;\eps\to 0\;\;}\quad q_0-f_0=:p_0\,.
\end{align*}
in $\fspaceL^\infty\at{K}$ and pointwise almost everywhere.
\end{corollary}
\begin{proof}
The first and the second statement have already been shown in Lemma \ref{konv: q} and \ref{lem: interface}. The claim concerning $f_\eps$ follows in view of the Arzel\`{a}-Ascoli theorem by combining the H\"older estimates for the essential fluctuations in Lemma \ref{lem: holder ess time} and \ref{lem: holder ess space} with the convergence of the negligible fluctuations in Lemma \ref{lem: n12} and \ref{lem: n34}. Finally, the convergence of $p_\eps$ is a direct consequence of \eqref{eq: iE} and \eqref{eq:IdentificationX}.
\end{proof}
To characterize the dynamics of any limit along subsequences we also introduce the macroscopic interface curve by
\begin{align*}
\Xi_0=\big\{\pair{t}{x}\in\ccinterval{0}{T}\times \Rset: x=\xi_0\at{t}\big\}
\end{align*}
and recall that this is the graph of a Lipschitz continuous function according to Lemma \ref{lem: interface}.
\begin{theorem}[\bf limit model along subsequences]
\label{theo: limit}\quad
\begin{enumerate}
\item \underline{\emph{weak formulation of the PDE\,:}}
We have
\begin{align}
\label{weak solution}
\begin{split}
\int\limits_0^{T}\int\limits_{\Rset}p_{0}\pair{t}{x}\,\partial_x^2\varphi\pair{t}{x}\dint{x}\dint{t}
=-\int\limits_0^{T}\int\limits_{\Rset}\Bat{p_{0}\pair{t}{x}+\sgn\bat{x-\xi_{0}\at{t}}}\partial_t\varphi\pair{t}{x}\dint{x}\dint{t}
\end{split}
\end{align}
for all $\varphi\in \fspaceC_\text{c}^{\infty}\bat{\oointerval{0}{T}\times\Rset}$. 
\item \underline{\emph{bounds for $p$\,:}}
The function $p_0$ is continuous and satisfies
\begin{align*}
-1\leq p_0\pair{t}{x}\leq+1\quad \text{for}\quad x\leq\xi_0\at{t}\,,\qquad\quad -1\leq p_0\pair{t}{x}\quad \text{for}\quad x\geq \xi_0\at{t}
\end{align*}
at any $t\in\ccinterval{0}{T}$.
\item \underline{\emph{hysteretic flow rule\,:}}
The estimate $\dot{\xi_0}\at{t}\leq 0$ hold for almost all $t\in \ccinterval{0}{T}$ and we have $\dot{\xi}_0\at{t}=0$  in case of $p_0\bpair{t}{\xi_0\at{t}}<1$.
\item \underline{\emph{further regularity\,:}}
The functions $f_0$ and $p_0$ admit a weak spatial derivative in $\fspaceL^2\bat{\ccinterval{0}{T}\times\Rset}$ and $\fspaceL^1_\loc\bat{\ccinterval{0}{T}\times\Rset}$, respectively.
\end{enumerate}
\end{theorem}
\begin{proof}
We consider convergent subsequences as in Corollary \ref{cor: conv} and recall that 
\begin{align}
\label{eq:UP}
p_\eps=u_\eps-\sgn\at{u_\eps}=u_\eps-\sgn\at{\cdot-\xi_\eps}
\end{align}
is a direct consequence of the single-interface property and the constitutive law  \eqref{eq: bilinear}. In what follows we adapt the argumentation of \cite[proof of Theorem 3.16]{HH13}.
\par\underline{\emph{Part 1}}\,:
We fix a test function $\varphi\in \fspaceC_\text{c}^{\infty}\bat{\oointerval{0}{T}\times\Rset}$ and suppose that $\eps>0$ is sufficiently small so that $\varphi\pair{t}{x}$ vanishes for $0\leq t\leq\eps$ and $T-\eps\leq t\leq T$.  Combining Scheme \ref{eq: IE} with \eqref{eq:UP} and elementary integral transformations  we find
\begin{align*}
\int\limits_0^{T}\int\limits_{\Rset}&p_{\eps}\pair{t}{x}\partial_x^2\varphi\pair{t-\eps^2}{x}\dint{x}\dint{t}\\&=
\int\limits_0^{T}\int\limits_{\Rset}p_{\eps}\pair{t+\eps^2}{x}\partial_x^2\varphi\pair{t}{x}\dint{x}\dint{t}
\\&=\int\limits_0^{T}\int\limits_{\Rset}\frac{u_{\eps}\pair{t+\eps^2}{x}-u_{\eps}\pair{t}{x}}{\eps^2}\varphi\pair{t}{x}\dint{x}\dint{t}\\
&=\int\limits_0^{T}\int\limits_{\Rset}\left[\frac{p_{\eps}(t+\eps^2,x)-p_{\eps}(t,x)}{\eps^2}+\frac{\sgn\bat{x-\xi_{\eps}\at{t+\eps^2}}-\sgn\bat{x-\xi_{\eps}\at{t}}}{\eps^2}\right]\varphi\pair{t}{x}\dint{x}\dint{t}\\
&=-\int\limits_0^{T}\int\limits_{\Rset}\left[p_{\eps}\pair{t}{x}+\sgn\at{x-\xi_{\eps}\at{t}}\right]\frac{\varphi\pair{t}{x}-\varphi\pair{t-\eps^2}{x}}{\eps^2}\dint{x}\dint{t}\,,
\end{align*}
where the second equality is just a reformulation of the update rule $u^n\to u^{n+1}$. The pointwise convergence of $p_\eps$ and $\xi_\eps$ yield the weak formulation \eqref{weak solution} in the limit $\eps\to0$ thanks to the Dominated Convergence Theorem. 
\par\underline{\emph{Part 2}}\,: 
The continuity of $p_0$ follows from the continuity of $q_0$ and $f_0$. From Definition \ref{defi: zul}, Lemma \ref{eq:UP}, and Assumption \ref{Ann: initial data} we also infer that
\begin{align*}
-2\leq u_\eps\pair{t}{x}\leq0\quad \text{for}\quad x\leq\xi_\eps\at{t}\,,\qquad 0\leq u_\eps\pair{t}{x}\quad \text{for}\quad x\geq\xi_\eps\at{t}
\end{align*}
hold for all admissible $t$, $x$, and $\eps$, so the claim now follows for any given $\pair{t}{x}\notin\Xi_0$ from the pointwise convergence in Corollary \ref{cor: conv} thanks to \eqref{Ann: initial data}. For $\pair{t}{x}\in\Xi_0$, we use the continuity of $p_0$.
\par\underline{\emph{Part 3}}\,: 
The function $\xi_0$ is Lipschitz continuous and its weak derivative is non-negative by Lemma~\ref{lem: interface}. We now fix an arbitrary point $\pair{t_*}{x_*}$ with
\begin{align*}
0<t_*\leq T\,,\qquad x_*=\xi_0\at{t_*}\,,\qquad p_0\pair{t_*}{x_*}<1
\end{align*}
and aim to show that $\xi_0$ is constant on a time interval $\ccinterval{t_0}{t_*}$ of positive length since this implies $\dot\xi_0\at{t_*}=0$. Without loss of generality we can assume $x_*<\xi_0\at{0}$ because otherwise we choose $t_0=0$ and are done. For any small $\eps$, we denote by $\widehat{t}_\eps$ the largest multiple of $\eps^2$ that is smaller or equal than $t_*$ and observe that
\begin{align*}
\widehat{x}_\eps:=\xi_\eps\nat{\widehat{t}_\eps}\quad\xrightarrow{\;\;\eps\to0\;\;}\quad x_*
\end{align*}
follows via
\begin{align*}
\babs{\widehat{x}_\eps-x_*}\leq \abs{\xi_\eps\at{\widehat{t}_\eps}-\xi_0\at{\widehat{t}_\eps}}+\abs{\xi_0\at{\widehat{t}_\eps}-\xi_0\at{t_*}}
&\leq \bnorm{\xi_\eps-\xi_0}_\infty+C\,\al\,\eps^2
\end{align*}
from the uniform convergence in Corollary \ref{cor: conv} and the Lipschitz continuity of $\xi_0$. Starting in $\widehat{t}_\eps$ we go back to the last time update $n_\eps\to n_\eps+1$, where the interface $\xi_\eps$ has moved, see Figure \ref{proof: hyst} for an illustration. We set $t_\eps:=\at{n_\eps+1}\,\eps^2$ and find
\begin{align}
\label{eq:peps}
 x_\eps:=\xi_\eps\at{t_\eps}=\xi_\eps\nat{\widehat{t}_\eps}=\widehat{x}_\eps\,,\qquad \quad p_\eps\pair{t_\eps}{x_\eps}=1\,,
\end{align}
where the second formula holds since the update rule in Scheme \ref{eq: IE} guarantees   
\begin{align*}
u^{n_\eps+1}\bat{\xi^{n_\eps+1}-0}=0\,,\qquad u^{n_\eps+1}\bat{\xi^{n_\eps+1}+0}=2
\end{align*}
for a left-moving interface. By passing to a not-relabeled subsequence we can assume that $t_\eps$ converges as $\eps\to0$ to a time $t_0\in\ccinterval{0}{T}$ and similarly to above we show 
\begin{align*}
\babs{\xi_\eps\at{t_\eps}-\xi_0\at{t_0}}\leq 
\babs{\xi_\eps\at{t_\eps}-\xi_0\at{t_\eps}}+
\babs{\xi_0\at{t_\eps}-\xi_0\at{t_0}}\quad\xrightarrow{\;\;\eps\to0\;\;}\quad 0\,.
\end{align*}
This implies
\begin{align*}
x_0:=\xi_0\at{t_0}=x_* 
\end{align*}
thanks to our choices of $\widehat{t}_\eps$ and $t_\eps$ and we conclude that the non-increasing function $\xi_0$ is in fact constant on the interval $\ccinterval{t_0}{t_*}$. Moreover, by \eqref{eq:peps} we get
\begin{align*}
\babs{p_0\pair{t_\eps}{x_\eps}-1}=\babs{p_0\pair{t_\eps}{x_\eps}-p_\eps\pair{t_\eps}{x_\eps}}\leq\bnorm{p_0-p_\eps}_\infty\quad\xrightarrow{\;\;\eps\to0\;\;}\quad 0\,,
\end{align*}
so  $t_0\neq t_*$ follows via 
\begin{align*}
p_0\pair{t_0}{x_*}=1>p\pair{t_*}{x_*}
\end{align*}
by our assumption on $t_*$.
\par\underline{\emph{Part 4}}\,: 
The claim follows by standard arguments from the convergence in Corollary \ref{cor: conv},  the uniform bounds in Lemma \ref{lem: reg}, and the regularity of $q_0$.
\end{proof}
%
%
\begin{figure}[ht!]
\centering{
\includegraphics[width=0.45\textwidth]{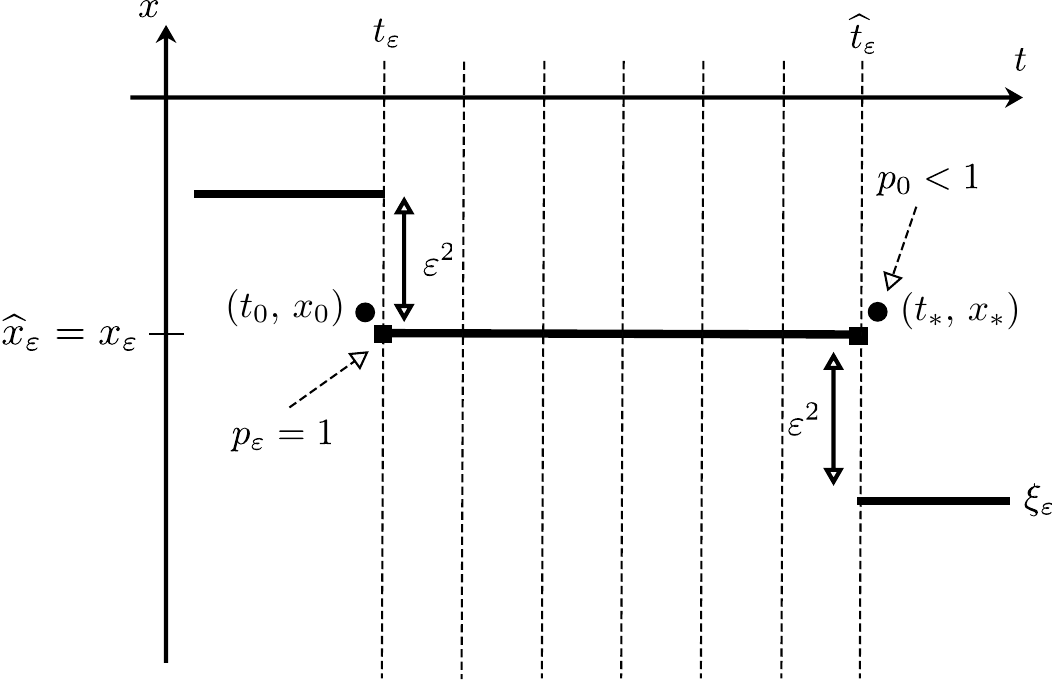}
}
\caption{The justification of the flow rule in the proof of Theorem \ref{theo: limit} uses that the interface curve $\xi_0$ must be constant on certain intervals of positive length.}
\label{proof: hyst}
\end{figure}
%
%
\begin{remarks*}\quad
\begin{enumerate}
\item 
The pointwise convergence
\begin{align*}
u_\eps\pair{t}{x}=p_\eps\pair{t}{x}+\sgn\bat{x-\xi_\eps\at{t}}\quad\xrightarrow{\;\;\eps\to 0\;\;}\quad p_0\pair{t}{x}+\sgn\bat{x-\xi_0\at{t}}=:u_0\pair{t}{x}
\end{align*}
ensures that $u_0\pair{t}{\cdot}$ exhibits for any $t\in\ccinterval{0}{T}$ a single phase interface at $x=\xi_0\at{t}$. In combination with the 
continuity of $p_0$ we thus obtain
\begin{align*}
\bjump{u_0\pair{t}{\cdot}}_{x=\xi_0\at{t}}=2
\end{align*}
as well as the second part of the Stefan condition \eqref{eq: simple RH}.
\item 
Both $u_0$ and $p_0$ satisfies the linear diffusion equation \eqref{eq: simple bulk} inside the bulk since 
\begin{align*}
\partial_t u_0\pair{t}{x}=\partial_t p_0\pair{t}{x}=\partial_x^2 p_0\pair{t}{x}=\partial_x^2 u_0\pair{t}{x}\qquad\text{in} \quad\ccinterval{0}{T}\times\Rset\,\setminus\Xi_0\,,
\end{align*}
holds at least in the sense of distributions according to \eqref{weak solution}.
\item 
Since $p_0$ admits a weak derivative with respect to $x$, we also have
\begin{align*}
-\int\limits_0^{T}\int\limits_{\Rset}u_{0}\pair{t}{x}\partial_t\varphi\pair{t}{x}\dint{x}\dint{t}=-\int\limits_0^{T}\int\limits_{\Rset}\partial_x p_{0}\pair{t}{x}\,\partial_x\varphi\pair{t}{x}\dint{x}\dint{t}
\end{align*}
for any test function $\varphi$. This implies the first part of the Stefan condition \eqref{eq: simple RH} for any time $t$ at which all quantities are well-defined.
\item 
The claim in Theorem \ref{theo: limit} implies the second part of the flow rule \eqref{eq: simple hysteresis} and also the first one by contraposition.
\item 
The properties of the linear diffusion and the continuity of $f_0$ provide
\begin{align*}
p_0\pair{t}{x}=q_0\pair{t}{x}+f_0\pair{t}{x}\quad \xrightarrow{\;\;t\to 0\;\;}\quad p_\ini\at{x}+f_0\pair{0}{x}=p_\ini\at{x}\,,
\end{align*}
for any $x\in\Rset$, where the initial data of $f_0$ vanishing due to $f_{\eps,\,\ess_4}\pair{0}{x}=0$ and the convergence results for the fluctuations. Moreover, the convergence
\begin{align*}
\xi_0\at{t}\quad \xrightarrow{\;\;t\to 0\;\;}\quad \xi_\ini
\end{align*}
follows from the Lipschitz continuity of $\xi_0$.
\end{enumerate}
\end{remarks*}
\begin{theorem}[\bf uniqueness]
The solution of the limit model from Theorem \ref{theo: limit} is uniquely determined for given initial data. In particular, the limit from Corollary \ref{cor: conv} is unique and the stated convergence holds along the whole family $\eps\to 0$.
\end{theorem}
\begin{proof}
The details of the proof can be found in \cite{HH13} and employs  ideas from \cite{Hil89} und \cite{Vis06}.
\end{proof}
%
%
%
\subsection*{Acknowledgements}
%
\addcontentsline{toc}{section}{Acknowledgements}
This work has been supported by the German Research Foundation (DFG) by the individual grant HE 6853/3-1.
%
%
%
\newpage
\appendix
\addcontentsline{toc}{section}{Appendix}
\section*{Appendix}
The estimates \eqref{eq:EstQ} and \eqref{eq: key ingedrient} compare the convolutive powers of the exponentially decaying  kernel $g_\eps$ from \eqref{eq: g} with those of the heat kernel $h_\eps$ in \eqref{eq:HeatKernel1} and rely on the following inequalities.
\begin{lemma}[\bf auxiliary result]
\label{lem: aux}
Let $A_n$, $A_\infty$, $B_n: \Rset\to\Rset$ be defined by
\begin{align*}
A_n\at{s}:=\frac{1}{\at{\D1+n^{-1}\,s^2}^n}\,,\qquad A_\infty\at{s}:=\exp\bat{-s^2}\,,\qquad B_n\at{s}:=n\,\babs{A_n\at{s}-A_\infty\at{s}}\,.
\end{align*}
Then, there exist a constant $C$ independent of $n$ such that 
\begin{align}
\notag
0\leq \int\limits_{-\infty}^{\infty}B_n\at{s}\dint{s}\leq C\,,\qquad\quad 0\leq \int\limits_{-\infty}^{\infty}s^{-2}\,B_n\at{s}\dint{s}\leq C
\end{align}
holds for all $n\in\Nset$.
\begin{proof}
We observe that
\begin{align}
\label{eq:AppX}
\eta-\tfrac12\,\eta^2\leq\log\at{1+\eta}\leq\eta\,,\qquad\quad 0\leq \exp\at{\eta}-1\leq \eta\,\exp\at{\eta}
\end{align}
holds for all $\eta>0$ and the upper bound in \eqref{eq:AppX}$_1$ provides the positivity of $B_n$ via
\begin{align*}
A_n\at{s}=\exp\Bat{-n\,\log\bat{1+n^{-1}\,s^2}}\geq \exp\at{-s^2}=A_\infty\at{s}\,.
\end{align*}
The lower bound in \eqref{eq:AppX}$_1$ implies
\begin{align*}
0\leq A_n\at{s}-A_\infty\at{s}
&= %
\exp\Bat{-n\,\ln\bat{1+n^{-1}\,s^2}}-\exp\bat{-s^2}
\\&\leq %
\exp\Bat{-n\,\bat{n^{-1}\,s^2+\tfrac12\,n^{-2}\,s^4}}-\exp\bat{-s^2}
\\&=%
\exp\at{-s^2}\,\Bat{\exp\at{\tfrac{1}{2}\,n^{-1}\,s^4}-1}
\end{align*}
and combining this with the upper bound in \eqref{eq:AppX}$_2$ we get
\begin{align*}
\int\limits_{0}^{n^{1/4}}\,\at{1+s^{-2}}\,B_n\at{s}\dint{s}&\leq 
\int\limits_{0}^{n^{1/4}}\at{1+s^{-2}}\,n\,\exp\at{-s^2}\,\bat{\tfrac12\,n^{-1}\,s^4}\,\exp\at{\tfrac12\,n^{-1}\,s^4}\dint{s}
\\&\leq C \int\limits_{0}^{\infty}\at{s^4+s^2}\,\exp\at{-s^2}\dint{s}
\leq C\,.
\end{align*}
Moreover, using monotonicity arguments and the lower bound in \eqref{eq:AppX}$_1$ once again we find
\begin{align*}
\int\limits_{n^{1/4}}^{\infty}s^{-2}\,B_n\at{s}\dint{s}\leq 
\int\limits_{n^{1/4}}^{\infty}B_n\at{s}\dint{s}
&\leq %
n\int\limits_{n^{1/4}}^{\infty}\frac{1}{\bat{1+n^{-1}\,s^2}^n}\dint{s}
=%
n^{3/2}\int\limits_{n^{-1/4}}^{\infty}\frac{1}{\bat{1+z^2}^n}\dint{z}
\\&\leq
\frac{n^{3/2}}{\at{1+n^{-1/2}}^{n-1}}\,\int\limits_{0}^{\infty}\frac{1}{1+z^2}\dint{z}
\\&\leq C\, n^{3/2}\,\at{1+n^{-1/2}}^{-n}
=C\, n^{3/2}\,\exp\Bat{-n\ln\bat{1+n^{-1/2}}}
\\&\leq
C\, n^{3/2}\,\exp\Bat{-n\bat{n^{-1/2}-\tfrac12\,n^{-1}}} 
 \leq C\,n^{3/2}\,\exp\at{-n^{ 1/2}} 
 \\&\leq C\,.
\end{align*}
The claim now follows due to the evenness of $B_n$.
\end{proof}
\end{lemma}
Using Fourier transform as well as Lemma \ref{lem: aux} and the integral subsitution $s=\eps\,\sqrt{n}\,k$ we obtain  \eqref{eq: key ingedrient} via
\begin{align*}
\Bnorm{\underbrace{\at{g_{\eps}\ast\ldots\ast g_{\eps}}}_{\substack{\text{$n$ times}}}-\underbrace{\at{h_\eps\ast\ldots\ast h_\eps}}_{\substack{\text{$n$ times}}}}_\infty
&\leq 
C\,\Bnorm{\wh{\underbrace{\at{g_{\eps}\ast\ldots\ast g_{\eps}}}_{\substack{\text{$n$ times}}}}-\wh{\underbrace{\at{h_\eps\ast\ldots\ast h_\eps}}_{\substack{\text{$n$ times}}}}}_1
\leq%
C\,\Bnorm{\widehat{g}_\eps^{\,n}-\widehat{h}_\eps^{\,n}}_1
\\&=%
C\,\int\limits_{-\infty}^{\infty}\abs{\frac{1}{\at{1+\eps^2\,k^2}^n}-\exp\at{-n\,\eps^2\,k^2}}\dint{k}
\\&\leq%
\frac{C}{\eps\,\sqrt{n}}\,\int\limits_{-\infty}^{\infty}\babs{A_n\at{s}-A_\infty\at{s}}\dint{s}
=\frac{C}{\eps\,n^{3/2}}\,\int\limits_{-\infty}^{\infty}B_n\at{s}\dint{s}\leq \frac{C}{\eps\,n^{3/2}}\,.
\end{align*}
A similar computation combined with
\begin{align*}
\sup_{k\in\Rset} k^2\,\babs{\widehat{q}_\ini\at{k}}\leq C\,\bnorm{q_\ini^{\prime\prime}}_1
\end{align*}
yields
\begin{align*}
\Bnorm{{\underbrace{\at{g_{\eps}\ast\ldots\ast g_{\eps}}}_{\substack{\text{$n$ times}}}}\ast q_\ini-{\underbrace{\at{h_\eps\ast\ldots\ast h_\eps}}_{\substack{\text{$n$ times}}}}\ast q_\ini}_\infty
&\leq%
C\,\int\limits_{-\infty}^{\infty}\babs{A_n\at{\eps\,\sqrt{n}\,k}-A_\infty\at{\eps\,\sqrt{n}\,k}}\,\babs{\widehat{q}_\ini\at{k}}\dint{k}\\
&\leq%
C\,\bnorm{q_\ini^{\prime\prime}}_1\int\limits_{-\infty}^{\infty}k^{-2}\,\babs{A_n\at{\eps\,\sqrt{n}\,k}-A_\infty\at{\eps\,\sqrt{n}\,k}}\dint{k}\\
&=%
C\,\bnorm{q_\ini^{\prime\prime}}_1\frac{\eps}{\sqrt{n}}\,\int\limits_{-\infty}^{\infty}s^{-2}\, B_n\at{s}\dint{s}
\end{align*}
and hence \eqref{eq:EstQ}.
%
%
%
%
%
\addcontentsline{toc}{section}{References}
%
%

%
%
\end{document}